\documentclass{scrartcl}
\usepackage{footmisc}
\usepackage[margin=3cm]{geometry}
\usepackage[utf8]{inputenc}
\usepackage[utf8]{inputenc}
\usepackage{amsmath, amsthm, amssymb, amsfonts}
\usepackage{dsfont}
\usepackage[T1]{fontenc}
\usepackage[bibstyle=alphabetic, citestyle=alphabetic, maxnames=99,doi=false,isbn=false, url=false]{biblatex}  
\usepackage{pst-node}
\usepackage{tikz-cd} 
\usepackage[english]{babel} 
\usepackage{mathrsfs}
\usepackage{url}
\usepackage{setspace}
\usepackage{graphicx}
\usepackage{subcaption}
\usepackage{amsmath}
\usepackage{enumitem}
\usepackage{tcolorbox}
\usepackage{xcolor}
\usepackage{stmaryrd}
\usepackage{upgreek}
\usepackage{mathtools}
\usepackage{multicol}
\usepackage{hyperref}
\usepackage{soul}
\usepackage{youngtab}
\usepackage{ytableau}
\usepackage{float}
\floatstyle{plain}
\newfloat{program}{thp}{lop}
\floatname{program}{Program}

\theoremstyle{plain}
\newtheorem{thm}{Theorem}[section]
\newtheorem{lemma}[thm]{Lemma}
\newtheorem{prop}[thm]{Proposition}
\newtheorem{corollary}[thm]{Corollary}

\theoremstyle{definition}
\newtheorem{defi}[thm]{Definition}

\newtheorem{rmk}[thm]{Remark}
\newtheorem{example}[thm]{Example}

\newtheorem{conjecture}[thm]{Conjecture}

\DeclareMathOperator{\spec}{Spec}
\DeclareMathOperator{\Ima}{Im}
\DeclareMathOperator{\Hom}{Hom}
\DeclareMathOperator{\End}{End}
\DeclareMathOperator{\Aut}{Aut}
\DeclareMathOperator{\stab}{Stab}
\DeclareMathOperator{\soc}{soc}
\DeclareMathOperator{\Lie}{Lie}

\newcommand{\id}{\mathrm{id}}

\newcommand{\G}{\mathbb{G}}
\newcommand{\F}{\mathbb{F}}
\newcommand{\Z}{\mathbb{Z}}
\newcommand{\A}{\mathbb{A}}
\newcommand{\Pj}{\mathbb{P}}

\newcommand{\thistheoremname}{}
\newtheorem*{genericthm*}{\thistheoremname}
\newenvironment{namedthm*}[1]
  {\renewcommand{\thistheoremname}{#1}%
   \begin{genericthm*}}
  {\end{genericthm*}}

\definecolor{bibi}{rgb}{0.79, 0.08, 0.48}
\definecolor{mame}{rgb}{0.0, 0.5, 0.5}

\title{Infinitesimal rational actions}
\date{}
\author{Bianca Gouthier}

\bibliography{bib} 

\begin{document}

\maketitle

\begin{abstract}
    \textbf{Abstract:} For any finite $k$-group scheme $G$ acting rationally on a $k$-variety, if the action is generically free then the dimension of $\Lie (G)$ is upper bounded by the dimension of the variety. We show that this is the only obstruction when $k$ is a perfect field of positive characteristic and $G$ is infinitesimal commutative trigonalizable. We also give necessary conditions to have faithful rational actions of infinitesimal commutative trigonalizable group schemes on varieties, and (different) sufficient conditions in the unipotent case over a perfect field.
\end{abstract}

\makeatletter \renewcommand{\@dotsep}{10000} \makeatother

\thispagestyle{empty}

\section{Introduction}


Let $k$ be a field and $X$ be a $k$-scheme. The automorphism group functor $\Aut_X$ of $X$ that associates to every $k$-scheme $S$ the group of $S$-automorphisms $\Aut_S(X\times_kS)$ is not representable in general. This object has been extensively studied: it is known for example that if $X$ is proper then $\Aut_X$ is represented by a $k$-group scheme locally of finite type \cite{MatsOort}. If $K/k$ is a finite purely inseparable field extension, the automorphism group scheme $\Aut_K$ has been studied for example by \cite{Begueri} and \cite{Chase}. 

For $G$ a $k$-group scheme, there is a bijection between $G$-actions $G\times_kX\rightarrow X$ on $X$ and group functor homomorphisms $G\rightarrow \Aut_X$. If the $G$-action is faithful, then $G$ is a subgroup functor of $\Aut_X$. Studying faithful group scheme actions yields then information on representable subgroups of $\Aut_X$. When $Y$ is the generic point of a variety $X$ (separated, geometrically integral scheme of finite
type) and $G$ is a finite $k$-group scheme, to give a $G$-action on $Y = \spec (k(X))$ is equivalent to giving a rational $G$-action on $X$. Studying such faithful rational actions imparts then knowledge on the automorphism group functor $\Aut_K$ of separable finitely generated extensions $K/k$. 

When $K=k(t_1,\dots,t_n)$ is a purely transcendental extension of $k$, then $\Aut_K(k)$ coincides with the Cremona group $\mathrm{Cr}_n(k)=\mathrm{Bir}_k(\Pj^n_k)$ in dimension $n$, that is by definition the group of birational automorphisms of $\Pj^n_k$. The Cremona group has been deeply studied in characteristic zero, while it has been less investigated in positive characteristic (see for example the survey \cite{Dolgachev}). Dolgachev made the following conjecture for the Cremona group over a field of positive characteristic.

\begin{conjecture}
If $k$ is a field of characteristic $p>0$, the Cremona group $\mathrm{Cr}_n(k)$ does not contain elements of order $p^s$ for $s>n$
\cite[Conjecture 37]{Dolgachev}.
\end{conjecture} 

The conjecture is true for $n=1$ since $\mathrm{PGL}_2(k)\simeq \Aut_k(k(t))$ does not contain elements of order $p^2$ if $char(k)=p>0$. Moreover,  it was proven for $n=2$ \cite{Dolgachev1}. The conjecture can be rephrased in the following way: if there exists a faithful rational action of a finite commutative $p$-group $G$ on $\mathbb{P}^n_k$ then $p^n_G=0$, where $p_G$ is the multiplication by $p$ morphism on $G$. Indeed there is a natural bijective correspondence between faithful actions of a finite commutative $p$-group $G$ on $k(t_1,\dots,t_n)$ and faithful rational actions of the corresponding constant group scheme on $\Pj^n_k$.

In this paper we are interested in rational actions of infinitesimal group schemes. The analogous of Dolgachev's conjecture for infinitesimal commutative unipotent group schemes arises naturally in one of the following ways: if $k$ is a field of characteristic $p>0$ and $G$ is an infinitesimal commutative unipotent $k$-group scheme, if there exists a faithful rational $G$-action on $\mathbb{P}^n_k$, then $p_G^n=0$ (or maybe $V_G^n=0$, where $V_G$ is the Verschiebung morphism of $G$). Both options turn out not to be true. Indeed,  for example any curve admits faithful rational actions of the $p^n$-torsion $E[p^n]$ of a supersingular elliptic curve $E$ (since in this case $E[p^n]$ is an infinitesimal commutative unipotent $k$-group scheme with one-dimensional Lie algebra and thus Theorem \ref{mainthm} applies) but $V_{E[p^n]}\neq0$ and $p_{E[p^n]}\neq0$ if $n>1$. 
What is indeed true is that if there exists a faithful rational $G$-action on a $k$-variety $X$ of dimension $n$, then $V_{\ker(F_G)}^n=0$. More precisely:

\begin{prop}\label{actionkerF}
    Let $G$ be an algebraic $k$-group scheme with commutative Frobenius kernel and $X$ be a $k$-variety of dimension $n$. If there exists a faithful rational $G$-action on $X$, then $s=dim_k(\Lie (\ker(F_{G})^m))\leq n$ and $V_{\ker(F_{G})^u}^{n-s}=0$, where $\ker(F_G)^m$ is the maximal $k$-subgroup scheme of multiplicative type of $\ker(F_G)$ and $\ker(F_G)^u:=\ker(F_G)/\ker(F_G)^m$.
\end{prop}

 The inverse implication of Proposition \ref{actionkerF} does not always hold true, see Example \ref{counterexample}. In the diagonalizable case, these actions are well understood and the converse statement is known.
Moreover, we show that there exist faithful rational actions of any infinitesimal commutative unipotent group scheme $G$ defined over a perfect field on any variety of dimension $n$ if $V_G^n=0$ (Proposition \ref{suffcondfaithact}).
As a consequence, the converse of Proposition \ref{actionkerF} holds true over a perfect field for infinitesimal commutative unipotent $k$-group schemes of height one and we have the following characterization: 

\begin{corollary}
    Let $k$ be a perfect field of characteristic $p>0$, $G$ be an infinitesimal commutative unipotent $k$-group scheme of height one and $X$ be a $k$-variety of dimension $n$. There exists a faithful rational $G$-action on $X$ if and only if $V_G^n=0$.
\end{corollary}

We are more precisely interested in rational actions which are generically free. Indeed in positive characteristic not all faithful actions admit an open dense subset $U\subseteq X$ that is $G$-stable and such that the action of $G$ on $U$ is free. For any finite $k$-group scheme $G$ acting rationally on a $k$-variety $X$, if the action is generically free then the dimension of $\Lie (G)$ is upper bounded by the dimension of the variety. 
Our main result is the following Theorem, which proves that this bound is the only obstruction to the existence of generically free actions for infinitesimal commutative trigonalizable (see Remark \ref{thm+Brion}) group schemes over a perfect field. If $G$ is unipotent, we also show that any generically free rational action on $X$ of (any iterated of) the Frobenius kernel of $G$ extends to a generically free rational action of $G$ on $X$. The proof we give is constructive and enables one to explicitly write such actions. 

\begin{thm}\label{mainthm}
    Let $k$ be a perfect field of characteristic $p>0$ and $G$ be an infinitesimal commutative unipotent $k$-group scheme with Lie algebra of dimension $s$. Then for every $k$-variety $X$ of dimension $\geq s$ there exist generically free rational actions of $G$ on $X.$ Moreover, for any $r\geq1$, any generically free rational action of $\ker(F_G^r)$ on $X$ can be extended to a generically free rational action of $G$ on $X$.
\end{thm} 
The difficulty is to construct actions in low dimension, i.e. close to the dimension of $\Lie (G)$. Indeed,  it is not so difficult to construct actions in high dimension for any infinitesimal trigonalizable group scheme (see Corollary \ref{asymptotic}). Fakhruddin proved that if $G$ is infinitesimal and $Y$ is a normal projective curve with a rational action of $G$, if there exists a normal projective variety $X$ with an action of $G$ and a $G$-equivariant dominant rational morphism $X\dashrightarrow Y$, then $G$ acts on $Y$ by automorphisms \cite[Proposition 2.2]{Fakhruddin}. In particular, in the above situation, if $Y$ is the projective line and the action is faithful, then $G$ is a subgroup scheme of $\mathrm{PGL}_{2,k}$. By Theorem \ref{mainthm}, for all unipotent infinitesimal group schemes with one-dimensional Lie algebra there exist generically free rational actions on the projective line. Nevertheless, most of these group schemes are not contained in $\mathrm{PGL}_{2,k}$: indeed for $p>2$ the only infinitesimal unipotent subgroup schemes of $\mathrm{PGL}_{2,k}$ are isomorphic to $\alpha_{p^n}$ (see for example \cite[Proposition 2.2]{gouthier2024unexpectedsubgroupschemespgl2k} to which we refer also for the case $p=2$), while there are many more of them e.g. the $p$-torsion of a supersingular elliptic curve or see also Examples \ref{answerbrion}, \ref{noncommutative} and \cite{gouthier2026infinitesimal} for a complete description of these group schemes in the commutative case over an algebraically closed field.  As a consequence, most of these rational actions on the projective line are not induced by actions, defined everywhere, on projective normal varieties of higher dimension. 

Combining Theorem \ref{mainthm} and the diagonalizable case treated by Brion in \cite[\textsection3]{brion2022actions} (see Remark \ref{thm+Brion}) the converse of Proposition \ref{actionkerF} is true, over a perfect field, for infinitesimal commutative trigonalizable $k$-group schemes with Lie algebra of dimension upper bounded by the dimension of the variety (in particular, if $s=\dim_k(\Lie (\ker(F_{G})^d))$ and  $\dim_k(\Lie (G))\leq n$, then $V_{\ker(F_{G})^u}^{n-s}=0$). 

Notice that if an infinitesimal commutative unipotent $k$-group scheme $G$ with Lie algebra of dimension $n$ can be embedded in a smooth connected $n$-dimensional algebraic group $\mathcal{G}$, then $G$ acts generically freely on it (by multiplication). Brion asked \cite[\textsection1]{brion2022actions} if, already in the one-dimensional case, these are the only examples that arise and moreover if these group schemes are always commutative (see also \cite[Remark 2.10]{Fakhruddin}). Examples \ref{answerbrion} and \ref{noncommutative} answer to these questions by the negative. The former shows that there are generically free rational actions on curves of infinitesimal commutative unipotent group schemes that are not subgroups of a smooth connected one-dimensional algebraic group. The latter shows that there exist generically free rational actions of non-commutative infinitesimal group schemes on varieties.

We conclude this introduction by making the link between this work and the notion of essential dimension. Informally speaking, the \textit{essential dimension} of an algebraic object is an integer that measures its complexity. This notion was introduced by Buhler and Reichstein in \cite{BR} for finite groups and was then extended by Merkurjev for functors from the category of field extensions of a fixed base field $k$ to the category of sets \cite{BF}. For a $k$-group scheme $G$, its essential dimension $\mathrm{ed}_k(G)$ computes, roughly speaking, the number of parameters needed to define all $G$-torsors over all schemes over $k$. Tossici conjectured \cite[Conjecture 1.4]{Tossici} that if $k$ is a field of positive characteristic and $G$ is a finite commutative unipotent $k$-group scheme, then $\mathrm{ed}_k(G)\geq n_{V}(G)$ where $n_{V}(G)$ is the order of nilpotency of the Verschiebung morphism of $G$. The conjecture is known to be true for $n_V(G)=2$ after Fakhruddin \cite[{Theorem 1.1}]{Fakhruddin}. Our hopes are that Theorem \ref{mainthm} might lead to further progress in the proof of this conjecture in the infinitesimal case.

\vspace{1em}
\noindent\textbf{Outline of the paper.} In Section \ref{FGS} we recall some notions and results around finite (commutative) group schemes and we introduce the \textit{socle} of a finite group scheme. Moreover, we prove Proposition \ref{structureicu} giving a description of the Hopf algebra of an infinitesimal commutative unipotent group scheme over a perfect field, which plays  an important role in the proof of Theorem \ref{mainthm}. In Section \ref{actions} we recall the main definitions and results around (rational) actions of finite group schemes on varieties and on their algebraic counterpart given by module algebra structures. In Section \ref{nilpder} we focus on nilpotent derivations, which are often encountered when studying actions of infinitesimal group schemes (see Proposition \ref{infinitesimalactions} and Example \ref{alphap}). Moreover, we study $p$-bases of finite field extensions and we prove Corollary \ref{systemsol} describing when some systems of differential equations admit a solution. This result is another of the building blocks needed for the proof of Theorem \ref{mainthm}. 

In Section \ref{genfreeactions} we deal with generically free actions. In the first part we prove the existence part of Theorem \ref{mainthm} in the case of commutative trigonalizable group schemes of height one (Proposition \ref{frobkeraction}). We then proceed with the proof of the general case. Section \ref{faithactions} is devoted to Dolgachev's conjecture revisited for infinitesimal group schemes and, more generally, to studying faithful rational actions of infinitesimal group schemes. Proposition \ref{actionkerF} gives necessary (but not sufficient, see the counterexample \ref{counterexample}) conditions for the existence of faithful rational actions of infinitesimal commutative trigonalizable group schemes. Moreover, we show that there exist faithful rational actions of any infinitesimal commutative unipotent group scheme $G$ defined over a perfect field on any variety of dimension $n$ if $V_G^n=0$ (Proposition \ref{suffcondfaithact}). We finish the paper with an example that illustrates our results about faithful rational actions in the case of the connected part of the $p$-torsion of abelian varieties.

\vspace{1em}
\noindent\textbf{Acknowledgments.} This work was done during my PhD at the Institut de Math\'ematiques de Bordeaux carried out under the supervision of Dajano Tossici, to whom I am deeply grateful for his guidance and constant support. I would also like to thank Michel Brion for engaging discussions and valuable comments on the initial version of this work. I thank Matthieu Romagny for carefully reading my PhD thesis, leading to improvements also of this paper. I am thankful to Xavier Caruso for his keen interest in the project and for offering a distinct perspective that proved helpful in various aspects. Additionally, I thank Damien Robert for providing valuable inputs. I am also grateful to the anonymous referee for their comments and remarks contributing to a better exposition in this final version.

\section{Finite group schemes}\label{FGS}

Throughout the whole work, $k$ will denote a ground field of characteristic $p>0$ and $\overline{k}$ an algebraic closure of $k$. Moreover, for every $k$-algebra $R$ and $k$-scheme $X$, we denote by $X_R$ the $R$-scheme $X\times_{\spec(k)}\spec(R)$. By $k$-\textit{algebraic} scheme we mean a $k$-scheme of finite type and we call $k$-\textit{algebraic group} a $k$-algebraic group scheme. All the group schemes considered will be algebraic groups. By $k$-\textit{variety} we mean a separated, geometrically integral $k$-scheme of finite type and we call \textit{curve} any $k$-variety of dimension $1$. If $X$ is a $k$-variety of dimension $n$, then its function field $K=k(X)$ is a separable, finitely generated extension of $k$ of transcendence degree $n$. For $G=\spec(A)$ an affine $k$-group scheme represented by the Hopf algebra $A$, we denote by $\Delta\colon A\rightarrow A\otimes_kA$ its comultiplication and by $\varepsilon\colon A\rightarrow k$ its counit. For $G$ an affine $k$-group scheme, we also denote by $k[G]$ the Hopf algebra representing it. 

\begin{defi}[Absolute Frobenius]
    Let $X$ be a $k$-scheme. The \textit{absolute Frobenius morphism} $\sigma_X\colon X\rightarrow X$ acts as the identity map on the underlying topological space $|X|$ while on the sections of $\mathcal{O}_X$ over an open subset $U\subseteq X$ it acts as the map \begin{align*}
        \mathcal{O}_X(U)&\rightarrow\mathcal{O}_X(U),\\
        a&\mapsto a^p.
    \end{align*}
\end{defi}

\begin{defi}[Relative Frobenius]
    Let $X$ be a $k$-scheme and $X^{(p)}=X\times_{k,f}\spec(k)$ be the base change with respect to the Frobenius morphism $f\colon k\rightarrow k,c\mapsto c^p$ of $k$. The \textit{relative Frobenius morphism} $F_X\colon X\rightarrow X^{(p)}$ is defined by the diagram \begin{center}
        \begin{tikzcd}
            X\arrow[drr,"\sigma_X", bend left]\arrow[ddr, bend right]\arrow[dr, dashed,"F_X"]&&\\
            &X^{(p)}\arrow[r]\arrow[d]&X\arrow[d]\\
            &\spec(k)\arrow[r,"\sigma_{\spec(k)}"]&\spec(k).
        \end{tikzcd}
    \end{center}
    We will refer to the relative Frobenius morphism just as the \textit{Frobenius morphism}.
\end{defi}

\begin{rmk}\leavevmode
\begin{enumerate}
\item The assignment $X\mapsto F_X$ is functorial, compatible with fiber products and commutes with extension of the base field.
    \item If $X$ is a scheme over $\F_p$, then $X^{(p)}\simeq X$ and the relative Frobenius $F_X$ coincides with the absolute Frobenius $\sigma_X$. Moreover,  for any extension $k\supseteq\F_p$ we have that $X_k^{(p)}\simeq X_k$ and $F_{X_k}=\sigma_X\times\id_k.$
    \item  When $G$ is a $k$-group scheme, then $G^{(p)}$ is also a $k$-group scheme and the Frobenius morphism $F_G\colon G\rightarrow G^{(p)}$ is a homomorphism of group schemes \cite[II.\textsection7, 1]{DG}. If $F_G^n=0$ for some $n\geq1$, then $G$ is said to have \textit{height} $\leq n$ and its height is the nilpotency index $\mathrm{ht}(G)$ of $F_G$.
\end{enumerate}
\end{rmk}

\begin{prop}\label{Liu}
    For any $k$-variety $X$, the Frobenius twist $X^{(p)}$ is geometrically integral. Moreover the relative Frobenius $F_X\colon X\rightarrow X^{(p)}$ induces a finite field extension of function fields $k\left(X^{(p)}\right)\subseteq k(X)$ of degree $p^{\dim(X)}$ and an isomorphism of $k\left(X^{(p)}\right)$ with the composite of the fields $ k$ and $(k(X))^p.$
\end{prop}

\begin{proof}
    See \cite[Chapter 3, Corollary 2.27]{Liu}.
\end{proof}

\begin{defi}[Lie algebra]
    Let $G$ be an affine $k$-group scheme and denote by $I_G=\ker(\varepsilon)$ its augmentation ideal (where $\varepsilon$ is the counit map $\varepsilon\colon k[G]\rightarrow k$). We define the \textit{Lie algebra} of $G$ to be $\Lie (G)=\Hom_k(I_G/I_G^2,k).$ As a $k$-vector space $\Lie (G)$ is the Zariski tangent space of $G$ at the identity element $e_G$ and it has an additional structure of Lie algebra (see for example \cite[II.\textsection4, 4]{DG}).
\end{defi}

\begin{rmk}\label{liefrobkern}
    Let $G$ be a $k$-group scheme and $F_G\colon G\rightarrow G^{(p)}$ its Frobenius morphism. Then $\Lie (G)=\Lie (\ker(F_G))$ (see \cite[II.\textsection7, 1.4]{DG}).
\end{rmk}

\begin{defi}[Infinitesimal group scheme]
    A $k$-group scheme $G=\spec(A)$ is said to be \textit{infinitesimal} if its augmentation ideal $I_G=\ker(\varepsilon\colon A\rightarrow k)$ is nilpotent.
\end{defi}

Notice that non-trivial infinitesimal group schemes exist only over fields of positive characteristic: indeed, by Cartier's Theorem, in characteristic zero all algebraic groups are smooth.

\subsection{Finite commutative group schemes}

Let $G=\spec(A)$ be an affine commutative $k$-group scheme and $$F_A\colon A^{(p)}=A\otimes_{k,f}k\rightarrow A,\quad a\otimes x\mapsto xa^p$$ be the relative Frobenius morphism of $A$, where $f$ denotes the Frobenius morphism of $k.$
For any $k$-vector space $V$, consider the $k$-vector space of symmetric tensors of order $p$, $\left(V^{\otimes p}\right)^{S_p}\subseteq V^{\otimes p}$. Notice that, since $G$ is commutative, $A$ is cocommutative and thus we have that the map given by the comultiplication $A\rightarrow A^{\otimes p}$ factors via $\left(A^{\otimes p}\right)^{S_p}$: \begin{center}
    \begin{tikzcd}
        A\arrow[r]\arrow[d,dashed]&A^{\otimes p}\\
        \left(A^{\otimes p}\right)^{S_p}.\arrow[ru,hook]&
    \end{tikzcd}
\end{center}
Let $s\colon A^{\otimes p}\rightarrow \left(A^{\otimes p}\right)^{S_p},a_1\otimes\dots\otimes a_p\mapsto\sum_{\sigma\in S_p}a_{\sigma(1)}\otimes\dots\otimes a_{\sigma(p)}$ be the symetrization map. By \cite[IV.\textsection3, 4.1]{DG}, $\left(A^{\otimes p}\right)^{S_p}$ is the direct sum of $s(A^{\otimes p})$ and of the submodule generated by $\{a\otimes\dots\otimes a\}_{a\in A}.$ Moreover the canonical map \begin{align*}
    \left(A^{\otimes p}\right)^{S_p}/s(A^{\otimes p})&\rightarrow A\otimes_{k,f}k\\
    a\otimes\dots\otimes a&\mapsto a\otimes 1
\end{align*} is a bijection.

\begin{defi}[Verschiebung]
    The \textit{Verschiebung morphism} $V_A$ of $A$ is by definition the composite $$A\longrightarrow \left(A^{\otimes p}\right)^{S_p}\stackrel{\lambda_A}{\longrightarrow}A\otimes_{k,f}k=A^{(p)}$$
where $\lambda_A$ is the unique $k$-linear map sending $a\otimes\dots\otimes a\mapsto a\otimes 1$ for any $a\in A$. The \textit{Verschiebung morphism} $V_G\colon G^{(p)}\rightarrow G$ is the homomorphism of group schemes induced by $V_A$.
\end{defi}

Notice that $\lambda_A$ is well-defined for what was said previously.
The assignment $G\mapsto V_G$ is functorial, compatible with fiber products and commutes with extension of the base field.

\begin{rmk}\label{FrobVerrmk}
    For any $k$-algebra $B$ the multiplication morphism $\left(B^{\otimes p}\right)^{S_p}\rightarrow B$ is given by the composite $$\left(B^{\otimes p}\right)^{S_p}\stackrel{\lambda_B}{\longrightarrow}B^{(p)}\stackrel{F_B}{\longrightarrow}B$$ and for any $k$-linear morphism $\varphi\colon B\rightarrow C$ we have the commutative diagram \begin{center}
    \begin{tikzcd}
        \left(B^{\otimes p}\right)^{S_p}\arrow[r,"\lambda_B"]\arrow[d,"\varphi^{\otimes p}"]&B^{(p)}\arrow[d,"\varphi^{(p)}"]\\
        \left(C^{\otimes p}\right)^{S_p}\arrow[r,"\lambda_C"]&C^{(p)}.
    \end{tikzcd}
\end{center} For more details see \cite[IV.\textsection3, 4]{DG}: in loc. cit. the second fact is stated for morphisms of $k$-algebras but can actually be generalized for any $k$-linear morphism.
\end{rmk}

\begin{rmk}\label{homalgebra}
   Let $(A,\Delta,\varepsilon)$ and $(B,m,u)$ be respectively a coalgebra and an algebra over $k$. Then $\Hom_k(A,B)$ has a $k$-algebra structure with multiplication given by $$\phi\otimes\chi\mapsto m\circ\phi\otimes\chi\circ\Delta$$ and unit $$k\rightarrow\Hom_k(A,B),1\mapsto u\circ\varepsilon.$$
\end{rmk}

\begin{lemma}\label{Frobofmaps}
    Let $G=\spec(A)$ be an affine commutative $k$-group scheme, $B$ be a $k$-algebra and let $C$ denote the $k$-algebra of $k$-linear morphisms $\Hom_k(A,B)$. For every element $g\in C^{(p)}$, it holds that $$F_C(g)=F_B\circ g\circ V_A.$$
\end{lemma}

\begin{proof}
    Since $F_C$ is a morphism of $k$-algebras, it is enough to show the result for $g$ of the form $f\otimes1=f^{(p)}$ with $f\in C=\Hom_k(A,B)$. We then have $F_C(f\otimes1)=f^p$ and we thus wish to show that $$f^p=F_B\circ f^{(p)}\circ V_A.$$ Using the definition of multiplication of the algebra $C$ one sees that the power $f^p$ is equal to the composite
 $$A\stackrel{comult}{\longrightarrow}\left(A^{\otimes p}\right)^{S_p}\stackrel{f^{\otimes p}}{\longrightarrow}\left(B^{\otimes p}\right)^{S_p}\stackrel{mult}{\longrightarrow }B.$$ By Remark \ref{FrobVerrmk} we obtain the commutative diagram \begin{center}
    \begin{tikzcd}
        A\arrow[r,"comult"]\arrow[dr,"V_A"]&\left(A^{\otimes p}\right)^{S_p}\arrow[r,"f^{\otimes p}"]\arrow[d,"\lambda_A"]&\left(B^{\otimes p}\right)^{S_p}\arrow[d,"\lambda_B"]\arrow[r,"mult"]&B\\
        &A^{(p)}\arrow[r,"f^{(p)}"]&B^{(p)}\arrow[ur,"F_B"]&
    \end{tikzcd}
\end{center} and thus the statement.
\end{proof}

Recall that a \textit{finite $k$-group scheme} is a $k$-group scheme that is finite as a $k$-scheme and that the category of finite commutative group schemes over a field $k$ is abelian. The \textit{order}, denoted $o(G)$, of a finite $k$-group scheme $G=\spec(A)$ is the dimension of $A$ as a $k$-vector space. 

\begin{lemma}
    Let $G=\spec(A)$ be a finite (commutative) $k$-group scheme. Then the dual of $A$ as a $k$-vector space $$A^\vee=\Hom_k(A,k)$$ is a finite dimensional (commutative) $k$-Hopf algebra.
\end{lemma}

\begin{proof}
    \cite[II.\textsection1, 2.10]{DG}
\end{proof}

\begin{defi}[Cartier dual] Let $G$ be a finite commutative $k$-group scheme.
    We call \textit{Cartier dual} of $G$ the finite commutative $k$-group scheme $$G^\vee=\spec(A^\vee).$$
\end{defi}

\begin{rmk}
    When $G$ is a finite commutative $k$-group scheme, the Verschiebung morphism $V_G\colon G^{(p)}\rightarrow G$ coincides with the dual of the Frobenius morphism of the Cartier dual $G^\vee$, $F_{G^\vee}\colon (G^\vee)^{(p)}\simeq(G^{(p)})^\vee\rightarrow {G^\vee}$ (see \cite[IV.\textsection3, 4.9]{DG}).
\end{rmk}

\subsection{The socle of a finite group scheme}

In the context of group theory, the \textit{socle} of a finite abstract group $G$ is the subgroup generated by the non-trivial minimal normal subgroups of $G$. We introduce here a generalization to finite group schemes of this classical definition; for this idea we are thankful to Michel Brion. 

\begin{defi}[Socle]
    For $G$ a finite $k$-group scheme, we define the \textit{socle} of $G$, denoted $\soc(G)$, to be the $k$-subgroup scheme generated by the non-trivial minimal normal $k$-subgroup schemes of $G$.
\end{defi}

The following Lemma describes some properties of the socle of a finite group scheme.

\begin{lemma}\label{soc1}
   Let $G$ be a finite $k$-group scheme.
    \begin{enumerate}
        \item $G$ is non-trivial if and only if $\soc(G)$ is non-trivial.
        \item $\soc(G)$ is a normal $k$-subgroup scheme of $G$.
        \item $\soc(G)\times_GH$ is non-trivial for any non-trivial normal $k$-subgroup scheme $H$ of $G$.
        \item If $G$ is commutative, then $\soc(H)=\soc(G)\times_GH$ for any $k$-subgroup scheme $H\subseteq G$, in particular $\soc(\soc(G))=\soc(G)$. 
        \item If $G$ is infinitesimal, then $\soc(G)\subseteq \soc(\ker(F_G))$. If in addition $G$ is commutative, then $\soc(G)=\soc(\ker(F_G))$.
        \item If $G_1$ and $G_2$ are finite commutative $k$-group schemes, then $$\soc(G_1\times_kG_2)=\soc(G_1)\times_k\soc(G_2).$$
        \item For any morphism of finite commutative $k$-group schemes $G_1\rightarrow G_2$, the induced morphism $\soc(G_1)\rightarrow G_2$ factors via $\soc(G_2)$.
    \end{enumerate}
\end{lemma}

\begin{proof}\leavevmode
    \begin{enumerate}
     \item Since $G$ is normal in itself and it is finite, there exist non-trivial minimal normal subgroup schemes. 
     \item  Clear by definition since the socle is generated by non-trivial normal subgroup schemes.
     \item Since $G$ is finite we may suppose that $H$ is minimal, hence $\soc(G)\times_GH=H$.
    \item Notice that since $G$ is commutative any of its subgroup schemes is normal. A non-trivial minimal $k$-subgroup scheme of $H$ is also a minimal $k$-subgroup scheme of $G$. Therefore $\soc(H)\subseteq \soc(G)\times_GH$. Let $N$ be a non-trivial minimal $k$-subgroup scheme of $G$. By minimality, either $N$ is a $k$-subgroup scheme of $H$ or $N\times_GH$ is trivial. Suppose the former. Then $N$ is also a minimal $k$-subgroup scheme of $H$. Therefore the equality. 
    \item Let $N$ be a non-trivial minimal normal $k$-subgroup scheme of $G$. Since $G$ is infinitesimal, then $N\times_G\ker(F_G)$ is a non-trivial normal $k$-subgroup scheme of $G$. Therefore, by minimality, $N$ is a $k$-subgroup scheme of $\ker(F_G)$. Hence $\soc(G)\subseteq \soc(\ker(F_G))$. If $G$ is commutative, by the previous point also the other inclusion holds. 
    \item Clearly $\soc(G_1\times_kG_2)$ is contained in $\soc(G_1)\times_k\soc(G_2).$ Take now $N_1\times_kN_2$ with $N_i$ non-trivial minimal $k$-subgroup scheme of $G_i$. Then $N_i$ is also a minimal $k$-subgroup scheme of $G_1\times_kG_2$ (notice that again we are using the assumption that the $G_i$'s are commutative). Therefore, by definition of the socle subgroup scheme, $N_1\times_kN_2\subseteq \soc(G_1\times_kG_2)$ and thus also the inverse inclusion holds true.
    \item Let $N$ be a non-trivial minimal $k$-subgroup scheme of $G_1$, then $N$ is mapped to a minimal $k$-subgroup scheme of $G_2$.
    \end{enumerate}
\end{proof}

\subsection{Trigonalizable group schemes}

\begin{defi}[Unipotent group scheme]
    A $k$-algebraic group $G$ is said to be \textit{unipotent} if it is isomorphic to an algebraic subgroup of the $k$-algebraic group of upper triangular unipotent matrices $U_n$ for some $n\geq1$.
\end{defi}

The group scheme of Witt vectors $W$ over a perfect field $k$ of positive characteristic $p$ plays a central role in the study of unipotent commutative $k$-group schemes. A reference for this is \cite[V.\textsection 1 and \textsection 4]{DG}. 
We denote by $W_n$ the $k$-group scheme of Witt vectors of length $\leq n$ and by $W_n^m$ the kernel of the morphism $F^m\colon W_n\rightarrow W_n$. Notice that if we want to consider $r$ copies of $W_n$ we will use the notation $(W_n)^r$ with the parenthesis. Recall that $W_n^m$ is the Cartier dual of $W_m^n$ for every $n,m\geq1$ (see \cite[III.\textsection4]{Demazure}).

\begin{prop}\label{infwitt}
If $k$ is perfect, then every infinitesimal commutative unipotent $k$-group scheme $G$ can be embedded in $\left(W_n^m\right)^r$ for some $n,m,r\geq1$.
\end{prop}

\begin{proof}
   See \cite[V.\textsection1, Proposition 2.5]{DG}.
\end{proof}

\begin{rmk}\label{FVnilp}
A finite commutative $k$-group scheme is infinitesimal unipotent if and only if its Frobenius and Verschiebung morphisms are both nilpotent (see \cite[IV.\textsection3, 5.3]{DG}). In particular, in Proposition \ref{infwitt} one can take $m$ and $n$ to be respectively their nilpotency indices (this is a direct consequence of the functoriality of the Frobenius and Verschiebung morphism). 
\end{rmk}

\begin{defi}[Diagonalizable group scheme/of multiplicative type]
    An affine $k$-group scheme $G$ is said to be \textit{diagonalizable} if it is represented by the group-algebra $k[M]$ for some abstract abelian group $M$, where the $k$-Hopf algebra structure is given by $\Delta\colon m\mapsto m\otimes m$ and $\varepsilon\colon m\mapsto1$ for every $m\in M$. It is said to be of \textit{multiplicative type}, if $G_{k^{sep}}$ is diagonalizable for some separable closure $k^{sep}$ of $k$.
\end{defi}

\begin{rmk}
    Notice that every diagonalizable group scheme is a finite product of copies of $\G_m$ and various $\mu_n$ (see \cite[Chapter 2.2]{Waterhouse}), where the latter is the subgroup scheme of $\G_m$ of $n$th roots of unity. 
\end{rmk}

\begin{defi}[Trigonalizable group scheme]
    A $k$-group scheme $G$ is said to be \textit{trigonalizable} if it is affine and it has a closed normal unipotent subgroup $G^u$ such that $G/G^u$ is diagonalizable (see for example \cite[IV.\textsection2, Definition 3.1]{DG}).
\end{defi}

Let us recall the Theorem of decomposition of commutative affine $k$-group schemes.

\begin{thm}\label{affinecommgrpschms}
    Let $G$ be a commutative affine $k$-group scheme. Then: 
    \begin{enumerate}[label=(\roman*)]
        \item $G$ has a maximal $k$-subgroup scheme $G^m$ of multiplicative type and $G/G^m$ is unipotent;
        \item if $k$ is perfect, $G$ has a maximal unipotent $k$-subgroup scheme $G^u$ and $G\simeq G^u\times_kG^m$. In particular, $G$ is trigonalizable if and only if $G^m$ is diagonalizable.
    \end{enumerate}
\end{thm}

\begin{proof}
    See \cite[IV.\textsection3, Theorem 1.1]{DG}.
\end{proof}

Before giving the following definitions, let us recall that $\End(\alpha_p)=k$ and $\End(\mu_p)=\F_p$ (see \cite[Chapter 2.d, 2.21, 2.22]{Milne}), where $\alpha_p$ and $\mu_p$ denote respectively the Frobenius kernel of the additive group scheme $\G_a$ and of the multiplicative group scheme $\G_m$. As a consequence, we have that $\Hom(\alpha_p,G)$ and $\Hom(\mu_p,G)$ have, for any $k$-group scheme $G$, a natural structure respectively of $k$-vector space and of $\F_p$-vector space. Here $\End$ and $\Hom$ refer to endomorphisms and homomorphisms of group schemes.


\begin{defi}[$a$-number]
    Let $k$ be perfect and $G$ be a commutative affine $k$-group scheme. The \textit{$a$-number} of $G$ is $\dim_k(\Hom(\alpha_p,G))$.
\end{defi}

\begin{rmk}\label{a-number}
    Notice that the $a$-number of $G$ coincides also with the maximal natural number $r$ such that $G$ contains a $k$-subgroup scheme isomorphic to $\alpha_p^r$. In addition to that, the $a$-number of $G$ is zero if and only if $G$ is of multiplicative type, since any non-trivial unipotent group scheme contains a $k$-subgroup scheme isomorphic to $\alpha_p$.
\end{rmk}

\begin{defi}[$p$-rank]
    Let $k$ be perfect and $G$ be a commutative trigonalizable $k$-group scheme. The \textit{$p$-rank} of $G$ is $\dim_{\F_p}(\Hom(\mu_p,G))$.
\end{defi}

\begin{rmk}
    Notice that the $p$-rank of $G$ coincides also with the maximal natural number $n$ such that $G$ contains a $k$-subgroup scheme isomorphic to $\mu_p^n$. In addition to that, the $p$-rank of $G$ is zero if and only if $G$ is unipotent.
\end{rmk}

\begin{lemma}\label{kerF}\label{goingdown}
    If $k$ is perfect and $G$ is an infinitesimal commutative unipotent $k$-group scheme, then the following are equivalent:
    \begin{enumerate}[label=(\roman*)]
        \item $r$ is the $a$-number of $G$;
        \item $r$ is the minimal natural number such that for any closed immersion $G\subseteq(W_n^m)^s$ there exists a projection $(W_n^m)^s\twoheadrightarrow(W_n^m)^r$, which forgets $s-r$ copies of $W_n^m$, inducing an immersion of $G$ in $(W_n^m)^r$.
    \end{enumerate}
    Moreover the following facts hold true:
    \begin{enumerate}[label=(\alph*)]
    \item $\soc(G)$ is isomorphic to $\alpha_p^r$ and it is the maximal $k$-subgroup scheme of $\ker(F_G)$ with trivial Verschiebung; 
        \item $r\leq \min\left(\dim_k(\Lie(G)),\dim_k(\Lie(G^\vee)\right)$ and $\dim_k(\Lie (G))=r$ if and only if $\ker(F_G)$ is isomorphic to $\alpha_p^r.$
    \end{enumerate}
    
\end{lemma}

\begin{proof}\leavevmode
\begin{enumerate}
    \item[$(i)\Rightarrow(ii)$] 
Let $r$ be the $a$-number of $G$, that is $r$ is the maximal natural number $r$ such that $G$ contains a $k$-subgroup scheme $H$ isomorphic to $\alpha_p^r$. By Proposition \ref{infwitt} there exists an embedding $G\subseteq(W_n^m)^s$ for some $s\geq1$. Notice that since $H\subseteq G$ is annihilated both by the Frobenius and the Verschiebung, then $$\alpha_p^r\simeq H\subseteq \left(W_1^1\right)^s=\alpha_p^s$$ and thus $s\geq r$. If $s=r$ we are done. Suppose that $s>r$, then there exists a projection $(W_n^m)^s\twoheadrightarrow(W_n^m)^{s-1}$ forgetting a copy of $W_n^m$ which induces an immersion $G\hookrightarrow(W_n^m)^{s-1}.$ Indeed, suppose that all the projections $$\pi_i\colon G\rightarrow(W_n^m)^{s-1}$$ have non-trivial kernel $\ker(\pi_i)$. Then $\ker(\pi_i)$ is a non-trivial $k$-subgroup scheme of $G$ for every $i=1,\dots,s$ and thus it contains a $k$-subgroup scheme isomorphic $\alpha_p$. 
Since each $\ker(\pi_i)$ lies in a different copy of $W_n^m$, we therefore have $s$ linearly independent homomorphisms $\alpha_p\hookrightarrow G$, contradicting the fact that the $a$-number of $G$ is $r<s$.
Now again, if $s-1=r$ we are done, otherwise we repeat the same reasoning until reaching $r$. Clearly $r$ is minimal for this property since $\alpha_p^{r}$ is not isomorphic to a $k$-subgroup scheme of $(W_n^m)^{r-1}$.
\item[$(ii)\Rightarrow(i)$] By minimality of $r$ all the projections $$\pi_i\colon G\rightarrow(W_n^m)^{r-1}$$ have non-trivial kernel $\ker(\pi_i)$. Then $\ker(\pi_i)$ is a non-trivial $k$-subgroup scheme of $G$ for every $i=1,\dots,r$ and thus it contains a $k$-subgroup scheme isomorphic to $\alpha_p$. Since each $\ker(\pi_i)$ lies in a different copy of $W_n^m$, we therefore have $r$ linearly independent homomorphisms $\alpha_p\hookrightarrow G$. Therefore $G$ contains a $k$-subgroup scheme isomorphic to $\alpha_p^r$ and clearly $r$ is maximal for this property since $\alpha_p^{r+1}$ is not isomorphic to a $k$-subgroup scheme of $(W_n^m)^{r}$.
\end{enumerate}

\begin{enumerate}[label=(\alph*)]
    \item Let $r$ be the $a$-number of $G$. Then there are $r$ linearly independent homomorphisms $\alpha_p\hookrightarrow G$ and each $\alpha_p$ is a minimal normal $k$-subgroup scheme of $G$. Therefore $\soc(G)=\soc(\ker(F_G))$ contains a $k$-subgroup scheme isomorphic to $\alpha_p^r.$ On the other hand, all the minimal normal subgroups of $\ker(F_G)$ are copies of $\alpha_p$ (see \cite[VI.\textsection2, Proposition 2.5]{DG}), thus $\soc(\ker(F_G))=\alpha_p^s$ for some $s\geq 1$. Combining the previous inclusion and the maximality of $r$ we obtain the equality. Suppose that $H$ is a $k$-subgroup scheme of $\ker(F_G)$ with trivial Verschiebung. Then $H\subseteq\G_a^{s'}$ for some ${s'}\geq1$ (see \cite[IV.\textsection3, Theorem 6.6]{DG}) and by the first point we can suppose that ${s'}$ is the maximal natural number such that $H$ contains a $k$-subgroup scheme isomorphic to $\alpha_p^{s'}$. Since $H\subseteq\ker(F_G)$, then $H=\ker(F_H)$. Moreover $\dim_k(\Lie (H))=s'$ and thus by order reasons we have $H=\ker(F_H)\simeq\alpha_p^{s'}$. By maximality of $r$, $H\subseteq \soc(G)\simeq\alpha_p^r$.
    \item Let $r$ be the $a$-number of $G$, then $\dim_k(\Lie(G))\geq r.$ By assumption $G$ contains a $k$-subgroup scheme $H$ isomorphic to $\alpha_p^r$. Dualizing, we obtain the faithfully flat homomorphism $$G^\vee\twoheadrightarrow H\simeq\alpha_p^{r}$$ and thus ${r}\leq \dim_k(\Lie(G^\vee))$. The inequality is given by \cite[Proposition 2.5]{BrochardMezard} which states that if we have a flat local morphism $A\rightarrow B$ of Noetherian local rings with maximal ideal and residue field respectively $m_A,m_B$ and $\kappa(A),\kappa(B)$, then $\dim_{\kappa(A)}\left(m_A/m_A^2\right)\leq \dim_{\kappa(B)}\left(m_B/m_B^2\right)$. For the last statement, clearly if $\ker(F_G)\simeq\alpha_p^r$ then $\dim_k(\Lie (G))=r.$ Now, by assumption we have $\alpha_p^r\simeq H\subseteq G$, so in particular $\alpha_p^r\simeq H\subseteq\ker(F_G)$, and if $\dim_k(\Lie (G))=r$ the equality must hold since in this case $H$ and $\ker(F_G)$ have both order $p^r$.
\end{enumerate} 
\end{proof}

\begin{lemma}\label{diag}
    Let $G$ be an infinitesimal commutative $k$-group scheme.
    \begin{enumerate}
        \item If $\ker(F_G)$ is diagonalizable, then $\soc(G)=\ker(F_G)=\mu_p^n$, where $n$ is the $p$-rank of $G$. 
\end{enumerate}
         Moreover, if $k$ is perfect and $G$ is trigonalizable:
         \begin{enumerate}
        \item[2.] $\soc(G)\simeq\alpha_p^r\times_k\mu_p^n$, where $r$ is the $a$-number of $G$ and $n$ is the $p$-rank of $G$. In particular,  $$\soc(G)=\left(\ker(F_{G})\times_{G}\ker(V_{G^{(1/p)}})\right)\times_k\ker(F_{G/G^u}).$$
        \item[3.] $\soc(G)\times_kK= \soc(G_K)$ for any field extension $K/k$.
    \end{enumerate}
\end{lemma}

\begin{proof}\leavevmode
\begin{enumerate}
    \item By assumption $\ker(F_G)=\mu_{p}^n$ where $n$ is the maximal natural number such that $\mu_p^n\subseteq G$. Then, by Lemma \ref{soc1}, $\soc(G)=\soc(\ker(F_G))=\mu_{p}^n$.
\item Since $k$ is perfect, then $G\simeq G^u\times_kG/G^u$ and by Lemma \ref{soc1} $$\soc(G)\simeq\soc(G^u)\times_k\soc(G/G^u).$$ Therefore the first part of the statement follows by 1. and Lemma \ref{goingdown}. We have already proved that $\soc(G/G^u)=\ker(F_{G/G^u})$. It is then enough to prove that $$\alpha_p^r\simeq\soc(G^u)=\ker(F_G)\times_G\ker(V_{G^{(1/p)}}).$$ The left to right inclusion is clear. By Lemma \ref{goingdown}, $G^u\subseteq(W_n^m)^r$. Hence, $$\ker(F_G)\times_G\ker(V_{G^{(1/p)}})\subseteq (W_n^1)^r\times_{(W_n^m)^r}(W_1^m)^r=(W_1^1)^r=\alpha_p^r.$$ The claimed equality then holds.
        \item If $K$ is perfect, the statement is a direct consequence of the compatibility of Frobenius and Verschiebung kernels with respect to base change. For the general case, clearly $\soc(G)\times_kK\subseteq \soc(G_K)$. Let $K^{perf}$ be the perfect closure of $K$. Then we have $$\soc(G)\times_kK^{perf}\hookrightarrow \soc(G_K)\times_KK^{perf}\hookrightarrow \soc(G_{K^{perf}})$$ and the first and last term coincide. Therefore $\soc(G)\times_kK^{perf}\simeq \soc(G_K)\times_KK^{perf}$ and so the inclusion $\soc(G)\times_kK\subseteq \soc(G_K)$ is in fact an equality.
\end{enumerate}
\end{proof}

\begin{rmk}
    Let $k$ be perfect, $G$ be a commutative trigonalizable $k$-group scheme and $G^u$ be its maximal unipotent $k$-subgroup scheme. The \textit{$a$-number} of $G$ coincides with the dimension of $\Lie (\soc(G^u)).$
\end{rmk}

\begin{example}\label{exasoc}\leavevmode
\begin{enumerate}
    \item $\soc\left(\left(W_n^m\right)^s\right)=\alpha_p^s$ for all $n,m,s\geq1$ and $\soc(G)=\alpha_p$ for any non-trivial $G\subseteq W_n^m$ and the $a$-numbers are respectively $s$ and $1$. 
    \item Let $k$ be algebraically closed and $A$ be an 
    abelian variety of dimension $g$ defined over $k$. The $p$-torsion $A[p]$ is a finite commutative $k$-group scheme
annihilated by $p$ with rank $p
^{2g}$. The \textit{$p$-rank} of $A$ is $$f=\dim_{\F_p}(\Hom(\mu_p,A[p])).$$ The \textit{$a$-number} of $A$ is $$a=\dim_k\left(\Hom(\alpha_p,A[p])\right).$$ Let $A[p]^0$ be the identity component of $A[p]$ and $A[p]^{0,u}$ its unipotent part. It is known that $$A[p]=A[p]^{0,u}\times_k\mu_p^f\times_k(\Z/p\Z)^f$$ (see for example \cite[III.15]{MumfordAbVar}). Then $\soc(A[p])=\alpha_p^a\times_k\mu_p^f\times(\Z/p\Z)^f$ and the $a$-number of $A$ coincides with $\dim_k\left(\Lie \left(\soc\left(A[p]^{0,u}\right)\right)\right)$ or equivalently it is the maximal natural number $a$ such that $A[p]$ contains a $k$-subgroup scheme isomorphic to $\alpha_p^a$, as pointed out already in Remark \ref{a-number}.
\end{enumerate}
\end{example}

\subsection*{Young diagrams for commutative unipotent group schemes of height one}\label{sectionyoung}

Let $k$ be perfect and $G$ be a commutative unipotent $k$-group scheme of height one, then $G\simeq\prod_{i=1}^sW_{n_i}^1$ for some $s,n_i\geq1$ (see \cite[IV.\textsection2, 2.14]{DG}). Moreover, we may suppose that $n_1\geq\dots\geq n_s$. We can then encode any such group scheme by a Young diagram $\tau(G)$, namely the one of shape $(n_1,\dots,n_s)$. For example $$\tau(\alpha_p)=\yng(1),\quad \tau(W_3^1\times_k\alpha_p)=\yng(3,1),\quad \tau(W_2^1\times_kW_2^1)=\yng(2,2).$$

The following lemma lists some straightforward properties, we thus omit the proof.

\begin{lemma}\label{younglemma}\leavevmode
    \begin{itemize}
        \item Two commutative unipotent $k$-group schemes of height one are isomorphic if and only if their Young diagrams coincide.
        \item The first column of $\tau(G)$ coincides with $\tau(\soc(G))$.
        \item The first $n$ columns represent $\tau(\ker(V_G^n))$ and the length of the $n$th column corresponds to the maximal $r$ such that $G$ contains a $k$-subgroup scheme isomorphic to $\left(W_n^1\right)^r$.
        \item The dimension of the Lie algebra of $G$ and $\log_p(o(G))$ both coincide with the number of boxes of $\tau(G)$.
    \end{itemize}
\end{lemma}

Given $G_1,\dots,G_l$ commutative unipotent $k$-group schemes of height one, the smallest commutative unipotent $k$-group scheme $G$ of height one containing all of them corresponds to the smallest Young diagram containing $\tau(G_i)$ for all $i$. Explicitly, if $\tau(G_i)=(n_{1i},\dots,n_{s_ii})$ for some $s_i\geq 1$ and for $i=1,\dots,l$ then $\tau(G)=(n_1,\dots,n_s)$ where $s=\max\{s_1,\dots,s_l\}$ and $n_j=\max\{n_{j1},\dots,n_{jl}\}$ for every $j=1,\dots,s$. For example, if we take $G_1=W_3^1\times_k\alpha_p$ and $G_2=W_2^1\times_kW_2^1$, then $\tau(G)=\yng(3,2)$ and $G=W_3^1\times_kW_2^1$. 

\subsection*{Algebraic description of infinitesimal commutative unipotent group schemes}

In the last part of this section we give a description of the Hopf algebra of an infinitesimal commutative unipotent group scheme over a perfect field. This result will be crucial for the proof of Theorem \ref{mainthm}. Let $k$ be perfect and $G$ be a commutative unipotent $k$-algebraic group. Then $V_G^n=0$ for some nilpotency index $n\geq1$. We then have the cofiltration $$G=G/\Ima(V_G^n)\rightarrow G/\Ima(V_G^{n-1})\rightarrow\dots\rightarrow G/\Ima(V_G)\rightarrow0.$$ We call $G_i$ the $k$-group scheme $G/\Ima(V_G^i)$ and $H_i$ the kernel of the map $G_i\rightarrow G_{i-1}.$ Notice that then $H_i=\Ima(V_G^{i-1})/\Ima(V_G^i)$ and thus is killed by the Verschiebung. If $G$ is infinitesimal, then, by \cite[IV.\textsection3, Theorem 6.6]{DG}, $H_i\simeq\prod_{j=1}^{r_i}\alpha_{p^{l_{ij}}}$ for some $r_i\geq1$ and $l_{i1},\dots,l_{ir_i}\geq1.$ Moreover,  there are epimorphisms $H_{i}^{(p)}\rightarrow H_{i+1}$ which are induced by $V_G\colon \left(\Ima(V_G^{i-1})\right)^{(p)}\rightarrow \Ima(V_G^{i})$. In particular, the order and the dimension of the Lie algebra of the $H_i$'s are decreasing (the latter is given by \cite[Proposition 2.5]{BrochardMezard}, applied as explained at the end of the proof of Lemma \ref{goingdown}).


\begin{prop}\label{structureicu} 
Let $k$ be perfect, $G$ be an infinitesimal commutative
unipotent $k$-group scheme and $n$ be the nilpotency index
of $V_G$. For all $i=1,\dots,n$ let $G_i=G/\mathrm{Im}(V_G^i)$,
$H_i=\mathrm{Im}(V_G^{i-1})/\mathrm{Im}(V_G^i)$ and
$r_i=\dim_k(\mathrm{Lie}(H_i))$. Then, for all $i=1,\dots,n$, there exist integers $l_{i1},\dots,l_{ir_i}\ge 1$, a $k$-group scheme $\mathcal{G}_i$ and a commutative diagram 

\begin{center}
    \begin{tikzcd}
    
        & 0\arrow[d] &0\arrow[d]&0 \arrow[d]\\
    
        0\arrow[r]&H_{i}\arrow[d]\arrow[r]&G_i\arrow[d]\arrow[r]&G_{i-1}\arrow[d,"\id"]\arrow[r]&0\\
        0\arrow[r]&\G_a^{r_i}\arrow[r]\arrow[d,"\phi"]&\mathcal{G}_i \arrow[d]\arrow[r]&G_{i-1}\arrow[r]\arrow[d]&0\\
        0\arrow[r]&\G_a^{r_i}\arrow[d]\arrow[r,"\id"]&\G_a^{r_i}\arrow[d]\arrow[r]&0\\
        & 0 &0
    \end{tikzcd}
    \end{center}
with exact rows and columns,
where $\phi= (F_{\G_a}^{l_{ij}})$. Moreover, the $k$-Hopf algebra of $\mathcal{G}_i$ is $$k[\mathcal{G}_i]=k[G_{i-1}][T_{i1},\dots,T_{ir_i}]$$ with comultiplication extending that of $k[G_{i-1}]$ and such that $$\Delta(T_{ij})=T_{ij}\otimes1+1\otimes T_{ij}+R_{ij}$$ where $R_{ij}$ is an element of $k[G_{i-1}]\otimes_kk[G_{i-1}]$.
\end{prop}


\begin{proof}
    It is enough to prove the statement for $G=G_n$ and we set $r:=r_n$. By Proposition \ref{infwitt}, $G\subseteq\left(W_n\right)^s$ for some $s\geq1.$ We then have the following commutative diagram with vertical maps being closed immersions
    \begin{center}
    \begin{tikzcd}
        0\arrow[r]&H_{n}\arrow[d]\arrow[r]&G\arrow[d]\arrow[r]&G_{n-1}\arrow[d]\arrow[r]&0\\
        0\arrow[r]&\G_a^s\arrow[r]&\left(W_n\right)^s\arrow[r,"\pi"]&\left(W_{n-1}\right)^s\arrow[r]&0.
    \end{tikzcd}
    \end{center} Now, $$H_{n}=\Ima(V_G^{n-1})\simeq\prod_{j=1}^r\alpha_{p^{l_j}}\subseteq\G_a^r$$ for some $1\leq r\leq s$ and $l_1,\dots,l_r\geq1$, where the isomorphism is given by the fact that $H_n$ is an infinitesimal subgroup scheme of $\G_a^r$. By Lemma \ref{goingdown} there exists a projection $$\rho\colon (W_n)^s\twoheadrightarrow(W_n)^r$$ such that the composite $$H_n\hookrightarrow(W_n)^s\twoheadrightarrow(W_n)^r$$ is a monomorphism and thus $H_n\hookrightarrow\G_a^r$. Consider the commutative diagram given by the schematic images of $\rho$:
    \begin{center}
    \begin{tikzcd}
        0\arrow[r]&H_n\arrow[d]\arrow[r]&\rho(G)\arrow[d]\arrow[r]&\rho(G_{n-1})\arrow[d]\arrow[r]&0\\
        0\arrow[r]&\G_a^r\arrow[r]&\left(W_n\right)^r\arrow[r,"\pi"]&\left(W_{n-1}\right)^r\arrow[r]&0.
    \end{tikzcd}
    \end{center} Its vertical maps are closed immersions and they factor in the following way    \begin{equation}\label{exact sequence}
    \begin{tikzcd}
        0\arrow[r]&H_n\arrow[d]\arrow[r]&\rho(G)\arrow[d]\arrow[r]&\rho(G_{n-1})\arrow[d]\arrow[r]&0\\
        0\arrow[r]&\G_a^r\arrow[r]&\pi^{-1}(\rho(G_{n-1}))\arrow[r,"\pi"]&\rho(G_{n-1})\arrow[r]&0.
    \end{tikzcd}
    \end{equation} 
    Notice that $\pi^{-1}(\rho(G_{n-1}))\rightarrow\rho(G_{n-1})$ is naturally a $\G_a^r$-torsor and since $\rho(G_{n-1})$ is affine it is the trivial torsor, that is $\pi^{-1}(\rho(G_{n-1}))=\rho(G_{n-1})\times_k\A^r_k
    $ as $k$-schemes. Its structure of $k$-group scheme is given by the embedding $\pi^{-1}(\rho(G_{n-1}))\hookrightarrow\left(W_n\right)^r.$  In particular  $$k[\pi^{-1}(\rho(G_{n-1}))]=k[\rho(G_{n-1})][T_1,\dots,T_r],$$
     with comultiplication extending that of $k[\rho(G_{n-1})]$ and such that $$\Delta(T_i)=T_i\otimes 1+1\otimes T_i+R_{i}$$ where $R_{i}$ is an element of $k[\rho(G_{n-1})]\otimes_kk[\rho(G_{n-1})]$.
    Since $\rho(G)\rightarrow\rho(G_{n-1})$ and $G\rightarrow G_{n-1}$ are both $H_n$-torsors and $G\rightarrow\rho(G)$ is $H_n$-equivariant, the commutative diagram \begin{center}
        \begin{tikzcd}
            G\arrow[r]\arrow[d]&G_{n-1}\arrow[d]\\
            \rho(G)\arrow[r]&\rho(G_{n-1})
        \end{tikzcd}
    \end{center} is indeed a pull-back diagram. 
    Therefore, $$G=G_{n-1}\times_{\rho(G_{n-1})}\rho(G),$$
    which is a closed subgroup scheme of 
    $\mathcal{G}_n:=G_{n-1}\times_{\rho(G_{n-1})}
\pi^{-1}(\rho(G_{n-1})).$ Moreover $$k[\mathcal{G}_n]=k[G_{n-1}]\otimes_{k[\rho(G_{n-1})]}k[\pi^{-1}(\rho(G))]$$$$=k[G_{n-1}]\otimes_{k[\rho(G_{n-1})]}k[\rho(G_{n-1})][T_1,\dots,T_r]$$$$=k[G_{n-1}][T_1,\dots,T_r]$$ with comultiplication extending that of $k[G_{n-1}]$ and such that $$\Delta(T_i)=T_i\otimes 1+1\otimes T_i+R_{i}$$ where $R_{i}$ is an element of $k[G_{n-1}]\otimes_kk[G_{n-1}]$ as wished.
Pulling back the exact sequences in \eqref{exact sequence} and     
by the Snake Lemma, we have a zig-zag map as in the following diagram:
    \begin{center}
    \begin{tikzcd}
    &&0\arrow[d]\arrow[r]&0\ar[out=-10, in=150]{dddll}\arrow[d]&\\
        0\arrow[r]&H_n\arrow[d]\arrow[r]&G\arrow[d]\arrow[r]&G_{n-1}\arrow[d]\arrow[r]&0\\
        0\arrow[r]&\G_a^r\arrow[d]\arrow[r]&\mathcal{G}_n\arrow[d]\arrow[r,"\pi"]&G_{n-1}\arrow[r]\arrow[d]&0\\
        &\G_a^r/H_n\arrow[r]&Q\arrow[r]&0&
    \end{tikzcd}
    \end{center}
    which concludes the proof.
    \end{proof}

   \begin{rmk}\label{algebraic pov}
   Let us point out that the important result of the above Proposition is that the $k$-Hopf algebra structure of $$k[\mathcal{G}_i]=k[G_{i-1}][T_{i1},\dots,T_{ir_i}]$$ extends that of $k[G_{i-1}]$ and is such that $$\Delta(T_{ij})=T_{ij}\otimes1+1\otimes T_{ij}+R_{ij}$$ where $R_{ij}$ is an element of $k[G_{i-1}]\otimes_kk[G_{i-1}].$
    We also translate here the above statement at the level of algebras, since it will be useful in the proof of Theorem \ref{mainthm}.
    The short exact sequence $$0\rightarrow G_i\rightarrow \mathcal{G}_i\rightarrow \G_a^{r_i}\rightarrow0$$ corresponds to 
    \begin{align*}
        0\rightarrow(S_1,\dots,S_{r_i})&\rightarrow k[G_{i-1}][T_{i1},\dots,T_{ir_i}]\rightarrow k[G_i]\rightarrow0\\
        S_i&\mapsto P_i
    \end{align*}
\    where $k[\G_a^{r_i}]=k[S_1,\dots,S_{r_i}].$ Therefore $$k[G_i]=k[G_{i-1}][T_{i1},\dots,T_{ir_i}]/(P_1,\dots,P_{r_i})$$ where the polynomials $P_j$ are primitive elements of $k[G_{i-1}][T_{i1},\dots,T_{ir_i}]$. Notice moreover that the polynomials $P_j$ are  congruent to $T_{ij}^{p^{l_{ij}}}$ for some $l_{ij}\geq1$ modulo the augmentation ideal of $k[G_{i-1}]$ for every $j=1,\dots,r_i$ by the short exact sequence \begin{align*}
        0\rightarrow I_{G_{i-1}}\rightarrow k[G_{i-1}][T_1,\dots,T_{r_i}]&\rightarrow k[T_1,\dots,T_r]\rightarrow0\\
        P_j&\mapsto T_{ij}^{p^{l_{ij}}}.
    \end{align*}
    \end{rmk}  
    
    \begin{example}\label{concretexample}
        Let us illustrate also via an example the proof of Proposition \ref{structureicu}.
        Consider the $k$-group scheme $G=\alpha_p\times_kW_2^1$. Then $V_G$ has nilpotency index $2$ and  $G\subseteq\left(W_2\right)^2.$ We then have the commutative diagram with vertical maps being closed immersions
    \begin{center}
    \begin{tikzcd}
        0\arrow[r]&H_{2}\arrow[d]\arrow[r]&G\arrow[d]\arrow[r]&G_{1}\arrow[d]\arrow[r]&0\\
        0\arrow[r]&\G_a^2\arrow[r]&\left(W_2\right)^2\arrow[r,"\pi"]&\left(W_{1}\right)^2\arrow[r]&0
    \end{tikzcd}
    \end{center} where $$H_{2}=\Ima(V_G)\simeq\alpha_p\subseteq\G_a$$ and $G_1\simeq\alpha_p\times_k\alpha_p$. Notice that $H_2$ is the copy of $\alpha_p$ contained in $W_2^1\subseteq G$ given by the image of the Verschiebung (of $W_2^1$). As a consequence, the projection on the second factor $$\rho\colon (W_2)^2\twoheadrightarrow W_2$$ is such that the composite $$H_2\hookrightarrow(W_2)^2\twoheadrightarrow W_2$$ is a monomorphism. Taking the schematic images of $\rho$ we obtain the commutative diagram given by:
    \begin{center}
    \begin{tikzcd}
        0\arrow[r]&H_2\arrow[d]\arrow[r]&\rho(G)\arrow[d]\arrow[r]&\rho(G_{1})\arrow[d]\arrow[r]&0\\
        0\arrow[r]&\G_a\arrow[r]& W_2\arrow[r,"\pi"]&W_1\arrow[r]&0
    \end{tikzcd}
    \end{center} where $\rho(G)=W_2^1$ and $\rho(G_1)\simeq\alpha_p$. Its vertical maps are closed immersions and they factor in the following way    \begin{equation*}
    \begin{tikzcd}
        0\arrow[r]&H_2\arrow[d]\arrow[r]&\rho(G)\arrow[d]\arrow[r]&\rho(G_{1})\arrow[d]\arrow[r]&0\\
        0\arrow[r]&\G_a\arrow[r]&\pi^{-1}(\rho(G_{1}))\arrow[r,"\pi"]&\rho(G_{1})\arrow[r]&0.
    \end{tikzcd}
    \end{equation*} 
    Notice that  $\pi^{-1}(\rho(G_{1}))=\rho(G_{1})\times_k\A^1_k\simeq\alpha_p\times_k\A^1_k
    $ as $k$-schemes with structure of $k$-group scheme given by the embedding $\pi^{-1}(\rho(G_{1}))\hookrightarrow\left(W_2\right)^2.$  In particular  $$k[\pi^{-1}(\rho(G_{1}))]=k[\rho(G_{1})][T_1]=k[T_0]/(T_0^p)[T_1],$$
     with comultiplication extending that of $k[\rho(G_{1})]$ and such that $$\Delta(T_1)=T_1\otimes 1+1\otimes T_1+S_1(T_0\otimes1,1\otimes T_0).$$
    Now, as in the proof of Proposition \ref{structureicu}, we have $$G=G_{1}\times_{\rho(G_{1})}\rho(G),$$
    which is a closed subgroup scheme of 
    $\mathcal{G}_2:=G_{1}\times_{\rho(G_{1})}
\pi^{-1}(\rho(G_{1})).$ Moreover $$k[\mathcal{G}_2]=k[G_{1}]\otimes_{k[\rho(G_{1})]}k[\pi^{-1}(\rho(G))]$$$$=k[G_{1}]\otimes_{k[\rho(G_{1})]}k[\rho(G_{1})][T_1]$$$$=k[G_{1}][T_1]=k[S_0,T_0]/(S_0^p,T_0^p)[T_1]$$ with comultiplication extending that of $k[G_{1}]$ and such that $$\Delta(T_1)=T_1\otimes 1+1\otimes T_1+S_1(T_0\otimes1,1\otimes T_0)$$
and $$k[G]=\left(k[S_0,T_0]/(S_0^p,T_0^p)\right)[T_1]/(T_1^p).$$
    \end{example}

\section{Actions of finite group schemes}\label{actions}

The first part of this section is devoted to recalling the main definitions around (rational) actions of finite group schemes on varieties, with a focus on faithful and (generically) free actions. The second part is centered on their algebraic counterpart which is given by module algebra structures (see Definition \ref{modalg}).

\subsection{Actions and rational actions}

Let $G$ be a $k$-group scheme, $X$ be a $k$-scheme equipped with a $G$-action $G\times_kX\rightarrow X$ and $\rho\colon G\rightarrow \Aut_X$ be the corresponding group functor homomorphism.

\begin{defi}[Centralizer]
    For any closed $k$-subscheme $Y$ of $X$, the \textit{centralizer} $C_G(Y)$ of $Y$ in $G$ is the subgroup functor that associates to any $k$-scheme $S$ the set of $g\in G(S)$ inducing the identity on the $S$-scheme $Y\times_kS$. The kernel of $\rho$ is the centralizer $C_G(X)$ of $X$ in $G$.
\end{defi}

\begin{defi}[Faithful action]
Let $G$ be a $k$-group scheme and $X$ be a $k$-scheme equipped with a $G$-action $\rho\colon G\rightarrow \Aut_X$. The $G$-action is said to be \textit{faithful} if its kernel is trivial.
\end{defi}

\begin{thm}
    Let $G$ be a $k$-group scheme acting on a separated $k$-scheme $X$. The centralizer $C_G(Y)$ of any closed $k$-subscheme $Y$ of $X$ is represented by a closed $k$-subgroup scheme of $G$.
\end{thm}

\begin{proof}
    See \cite[VI$_\mathrm{B}$, Example 6.2.4.e)]{SGA3}.
\end{proof}

\begin{lemma}\label{bschngenfaithful}
   Let $G$ be a $k$-group scheme and $X$ be a separated $k$-scheme endowed with a $G$-action. The $G$-action is faithful if and only if the induced $G_{\overline{k}}$-action on $X_{\overline{k}}$ is faithful.
\end{lemma}

\begin{proof}
The $G_{\overline{k}}$-action on $X_{\overline{k}}$ is faithful if and only if $C_{G_{\overline{k}}}(X_{\overline{k}})\simeq C_G(X)_{\overline{k}}$ is trivial and this holds true if and only if $C_G(X)$ is trivial.
\end{proof}

\begin{defi}[Free action]
    Let $G$ be a finite $k$-group scheme and $X$ be a $k$-scheme equipped with a $G$-action $\rho\colon G\times_kX\rightarrow X.$ Let $x\colon \spec(k(x))\rightarrow X$ be a point of $X$ and consider the composite $\nu\colon G\times_k\spec(k(x))\stackrel{\id\times x}{\longrightarrow}G\times_kX\stackrel{\rho\times\id}{\longrightarrow}X\times_kX$. The \textit{stabilizer} $\stab_G(x)$ of the point $x$ is the pull-back of the diagram \begin{center}
         \begin{tikzcd}
         \stab_G(x)\arrow[d,dashed]\arrow[r,dashed]&G_{k(x)}\arrow[d,"\nu"]\\
             \spec(k(x))\arrow[r,"diag"]&X\times_kX
         \end{tikzcd}
     \end{center} where the bottom arrow is the diagonal morphism.
     The $G$-action is said to be \textit{free at $x\in X$} if $\stab_G(x)$ is trivial. The $G$-action is said to be \textit{free} if it is free at every point. We denote by $X_{fr}$ the subset of free points of $X$.
\end{defi}

\begin{rmk}\label{stabsgrp}\leavevmode
\begin{itemize}
\item The stabilizer $\stab_G(x)$ of a point $x
\in X$ is the fiber over $\spec(k(x))$ of the \textit{stabilizer} $\stab_G$ defined as the preimage of the diagonal under the graph morphism $$G\times_kX\rightarrow X\times_kX,\quad(g,y)\mapsto (g\cdot y,y)$$ and $\stab_G$ is a closed subgroup scheme of the $X$-scheme $G\times_kX$; in particular, $G$ being finite, the projection $\stab_G\rightarrow X$ is also finite. As a consequence $X_{fr}$, which is the set of points $x$ of $X$ with trivial stabilizer that is for which $\stab_G(x)\simeq\spec(k(x))$, is an open $G$-stable subset of $X$.
    \item Notice that, by universal property of pull-backs, if $H$ is a $k$-subgroup scheme of $G$, then $$\stab_H(x)=\stab_G(x)\times_{G_{k(x)}}H_{k(x)}.$$
\end{itemize}
    
\end{rmk}

\begin{prop}\label{propgenfreeactions}
    Let $G$ be a finite $k$-group scheme and $X$ be an irreducible $k$-scheme with a $G$-action. The following are equivalent:
    \begin{enumerate}
        \item $X_{fr}\neq\emptyset$;
        \item the generic point $\eta$ of $X$ belongs to $X_{fr}$;
        \item $X_{fr}$ is dense in $X.$
    \end{enumerate}
\end{prop}

\begin{proof}
    As we recalled above, $X_{fr}$ is an open $G$-stable subset of $X$. The statement is a direct consequence of this and of the fact that $X$ is irreducible.
\end{proof}

\begin{defi}[Generically free action] Let $G$ be a finite $k$-group scheme and $X$ be an irreducible $k$-scheme with a $G$-action.
    We say that the action is \textit{generically free} if it satisfies one of the above equivalent conditions.
\end{defi}

\begin{rmk} When $G$ is a finite constant group acting on a variety, if the action is faithful then it is automatically generically free. 
    This fails in general for $G$ a finite $k$-group scheme. For example the action $\alpha_p^2\times_k\A^1_k\rightarrow\A^1_k$ given by $(a,b)\cdot x\mapsto x+ax^p+b$ is faithful (there is no non-trivial $k$-subgroup of $\alpha_p^2$ acting trivially) but not generically free, in fact the stabilizer of the generic point $\eta$ is $\stab_G(\eta)=\spec\left(k(x)[S,T]/(x^pS+T,S^p,T^p)\right)$. Faithful actions coincide with generically free actions also for diagonalizable $k$-group schemes (this is known and we also give a proof in Corollary \ref{freeifffaith}). Moreover, we show that this property holds, for instance, for infinitesimal commutative unipotent subgroup schemes of the $k$-group scheme of Witt vectors (see Remark \ref{wittvectors}).
\end{rmk}

\begin{prop}\label{bschngenfree}
    Let $G$ be a finite $k$-group scheme and $X$ be a $k$-variety with a $G$-action. The $G$-action is generically free if and only if the induced $G_{\overline{k}}$-action on $X_{\overline{k}}$ is generically free.
\end{prop}

\begin{proof}
    Let $\eta\colon \spec(k(\eta))\rightarrow X$ be the generic point of $X$. Since $X$ is geometrically integral, then the generic point of $X_{\overline{k}}$ is the base change $\overline{\eta}\colon \spec\left(k(\eta)\otimes_k\overline{k}\right)\rightarrow X_{\overline{k}}$ (see for example \cite[Chapter 3, Corollary 2.14]{Liu}). Therefore, by general properties of the base change, we have that the stabilizer of $\overline{\eta}$ is $$\stab_{G_{\overline{k}}}(\overline{\eta})\simeq \stab_G(\eta)_{\overline{k}}\simeq \stab_G(\eta)\times_{\spec(k(\eta))}\spec\left(k(\eta)\otimes_k\overline{k}\right)$$ where $\stab_G(\eta)$ is the stabilizer of $\eta.$ Now, since $k(\eta)\hookrightarrow k(\eta)\otimes_k\overline{k}$ is faithfully flat (it is a field extension since $X$ is geometrically integral), then $\stab_{G_{\overline{k}}}(\overline{\eta})$ is trivial (i.e. isomorphic to $\spec\left(k(\eta)\otimes_k\overline{k}\right)$) if and only if $\stab_G(\eta)$ is trivial (i.e. isomorphic to $\spec(k(\eta))$), as wished.
\end{proof}


\begin{defi}[Rational action]
    Let $G$ be a $k$-group scheme and $X$ a $k$-variety. A \textit{rational action} of $G$ on $X$ is a rational map $\rho\colon G\times_kX\dashrightarrow X$ such that:
    \begin{enumerate}[label=(\roman*)]
        \item the rational map $(\pi_1,\rho)\colon G\times_kX\dashrightarrow G\times_kX, (g,x)\mathrel{\mapstochar\dashrightarrow}(g,g\cdot x)$ is birational;
        \item the following diagram commutes 
        \begin{center}
            \begin{tikzcd}
                G\times_kG\times_kX\arrow[d,"\id_G\times\rho",dashed]\arrow[r,"m\times \id_X"]&G\times_kX\arrow[d,dashed,"\rho"]\\
                G\times_kX\arrow[r,dashed,"\rho"]&X
            \end{tikzcd}
        \end{center}
        where $m\colon G\times_kG\rightarrow G$ denotes the multiplication morphism of $G$.
    \end{enumerate}
\end{defi}


\begin{rmk} When $G$ is a \textit{finite} $k$-group scheme, there is a bijection between rational actions of $G$ on a $k$-variety $X$ and $G$-actions on the generic point of $X$(see \cite[Corollary 3.4]{brion2022actions}).
\end{rmk}

\begin{defi}[Faithful rational action]
    Let $G$ be a finite $k$-group scheme and $X$ be a $k$-variety equipped with a rational action $\rho\colon G\times_kX\dashrightarrow X.$ We say that it is a \textit{faithful rational action} if the corresponding action on the generic point of $X$ is faithful.
\end{defi}

The following is a known result (see for example \cite[Section 2]{TossiciVistoli}), we include the proof for the sake of completeness. The proof we give can be deduced by \cite[Lemma 5.3]{brion2022actions} where the case of curves is treated. 

\begin{prop}\label{dimlie}
    Let $G$ be a finite $k$-group scheme and $X$ be a $k$-variety endowed with a generically free rational $G$-action. 
    Then $$\dim_k(\Lie (G))\leq\dim(X).$$
\end{prop}

\begin{proof}  Let $U$ be an open subset of $X$ on which the action is defined (see for example \cite[Proposition 3.2]{brion2022actions}). 
Suppose first that $char(k)=0$, then $G$ is étale, thus $$0=\dim_k(\Lie (G))=\dim (G)$$ and the statement follows.
    Suppose then that $char(k)=p>0$ and let $G_1$ be the kernel of the Frobenius morphism $F_G\colon G\rightarrow G^{(p)}$. Then $G_1$ is an infinitesimal $k$-subgroup scheme of $G$ and $\Lie (G)=\Lie (G_1)$. If the $G$-action on $U$ is generically free, then the same holds for $G_1$. We can thus suppose that $G=G_1$ and, by Proposition \ref{bschngenfree}, that $k$ is algebraically closed.
    By Proposition \ref{propgenfreeactions}, the $G$-action is generically free if and only if $X_{fr}$ is dense in $X$ and thus there exists (since $X$ is geometrically integral and $k=\overline{k}$) a smooth closed point $x\in U$ with trivial stabilizer $\stab_G(x).$ Since $G$ is infinitesimal, then also $\stab_G(x)$ is such, hence $\stab_G(x)$ is trivial if and only if $\Lie (\stab_G(x))$ is trivial. Now, $\Lie (\stab_G(x))$ is the kernel of the natural map $\Lie (G)\rightarrow T_xU$ (see \cite[III.\textsection2, 2.6]{DG}) and therefore if the action is generically free this map is an injection and thus the statement.
\end{proof}

\subsection{Actions of finite group schemes and module algebras}

Some references for this part are \cite[Chapter VII]{Sweedler} and \cite[Chapter 4]{Montgomery}.

\begin{defi}[Module algebra]\label{modalg} Let $A$ be a $k$-Hopf algebra, not necessarily commutative. We say that a $k$-algebra $B$ is a (left) \textit{$A$-module algebra} if:
\begin{enumerate}
    \item $B$ is a (left) $A$-module via $\psi\colon A\otimes_kB\rightarrow B, a\otimes b\mapsto a\cdot b;$
    \item the morphism $\eta\colon B\rightarrow\Hom_k(A,B),b\mapsto(a\mapsto a\cdot b)$ is a (unital) morphism of $k$-algebras.
\end{enumerate}
One can give the same definition on the right as well.
\end{defi}

The following remark is of key importance for the way in which we will view module algebras throughout this whole work.

\begin{rmk}[Property of compatibility with products]\label{initialrmk}
Notice that the first condition of the above definition is equivalent to giving a map $$v\colon A\rightarrow\mathrm{End}_k(B),a\mapsto(b\mapsto a\cdot b)$$ which is a morphism of $k$-algebras, while the second requirement corresponds to asking that $v$ satisfies the following properties:
\begin{equation}\label{property}
   \left\{\begin{array}{ll}
    v(a)(1)=\varepsilon(a) \\
    v(a)(fg)=m_B(v\otimes v\circ\Delta(a))(f\otimes g)
\end{array}\right.
\end{equation}
 for any $a\in A$ and $f,g\in B$. Here $\varepsilon$ denotes the counit of $A,$ $\Delta$ its comultiplication and $m_B$ the multiplication of $B.$ We will also refer to (\ref{property}) as the \textit{property of compatibility with products}.
The first statement is straightforward by the definition of $A$-module. For the second one, recall that $\Hom_k(A,B)$ has a structure of $k$-algebra (Remark \ref{homalgebra}) with multiplication given by the convolution product $$\phi\otimes\chi\mapsto m_B\circ\phi\otimes\chi\circ\Delta$$ and unit $$k\rightarrow\Hom_k(A,B),1\mapsto u\circ\varepsilon.$$ Moreover $\eta(b)(a)=v(a)(b).$ Therefore, $\eta$ is a morphism of (unital) $k$-algebras if and only if $$\eta(fg)=m_B\circ\eta(f)\otimes\eta(g)\circ\Delta\quad\mbox{and}\quad \eta(1)=1_{\Hom_k(A,B)}$$ for all $f,g\in B$ if and only if $$v(a)(fg)=\eta(fg)(a)=m_B\circ\eta(f)\otimes\eta(g)\circ\Delta(a)=m_B(v\otimes v\circ\Delta(a))(f\otimes g)$$ and $$a\cdot1=v(a)(1)=\eta(1)(a)=\varepsilon(a)$$ for any $a\in A.$ Finally, notice that, if we denote by $I:= \ker(\varepsilon)$ the augmentation ideal of $A,$ then by what we just showed it holds that $v(a)(1)=0$ for any $a\in I.$ 
\end{rmk}

Module algebras are very useful when studying the actions of finite $k$-group schemes, thanks to the following result.

\begin{prop}\label{actionsvsmodulealgstrctr}
    Let $G=\spec(A)$ be a finite $k$-group scheme and $X=\spec(B)$ be an affine $k$-scheme. There is a bijection between the set of right actions of $G$ on $X$ and the set of left $A^\vee$-module algebra structures on $B$. 
\end{prop}

\begin{proof}
To give a right action $X\times_kG\rightarrow X$ of an affine group scheme $G=\spec(A)$ on an affine scheme $X=\spec(B)$ is equivalent to giving a coaction $B\rightarrow B\otimes_kA$ (see for example \cite[Chapter 3]{Waterhouse}).
The claimed bijection is then obtained associating to any coaction $\rho\colon B\rightarrow B\otimes_kA$ the $A^\vee$-module algebra structure \begin{align*}
    v\colon A^\vee&\rightarrow\mathrm{End}_k(B)\\
    \alpha&\mapsto (B\stackrel{\rho}{\longrightarrow}B\otimes_kA\stackrel{\id_B\otimes\alpha}{\longrightarrow}B\otimes_kk\simeq B).
\end{align*}
    For more details see for example \cite[\textsection4.1]{Montgomery}.
\end{proof}

When dealing with infinitesimal group schemes, one can specialize Proposition \ref{actionsvsmodulealgstrctr} and prove that to give an action of these group schemes amounts to exhibiting a certain number of differential operators (see \cite[II.\textsection4, 5]{DG} for a definition of the algebra of differential operators on a scheme) satisfying some relations. We will denote by $\mathrm{Diff}_k(B)$ the $k$-algebra of differential operators on an affine $k$-scheme $X=\spec(B)$.

\begin{prop}\label{infinitesimalactions}
Let $G=\spec(A)$ be an infinitesimal $k$-group scheme and $X=\spec(B)$ be an affine $k$-scheme. There is a bijection between the set of right actions of $G$ on $X$ and the set of homomorphisms of $k$-algebras $v\colon A^\vee\rightarrow \mathrm{Diff}_k(B)$ such that
\begin{equation*}
v(\mu)(fg)=m_B(v\otimes v\circ\Delta(\mu))(f\otimes g)
\end{equation*} for any $\mu\in A^\vee$ and $f,g\in B$, where $\Delta$ and $m_B$ denote respectively the comultiplication of $A$ and the multiplication of $B$.
\end{prop}

\begin{proof}
    See \cite[II.\textsection4, Proposition 7.2]{DG}.
\end{proof}

\begin{example}\label{alphap}\leavevmode
\begin{enumerate}
   \item Consider the self-dual infinitesimal $k$-group scheme $\alpha_p=\spec(k[T]/(T^p))$ whose group structure is given by $$\Delta(T)=T\otimes1+1\otimes T.$$ To give an action of $\alpha_p$ on a $k$-scheme $X=\spec(B)$ is equivalent to giving a $k$-linear derivation $\partial\colon B\rightarrow B$ such that $\partial^p=0$.
\item Consider the purely transcendental extension $k(t)/k$. The algebra of differential operators $\mathrm{Diff}_k(k(t))$ is a $k(t)$-vector space with basis given by $\left\{\frac{\partial}{\partial t^i}\right\}$ where $$\frac{\partial}{\partial t^i}(t^r)=\left\{\begin{array}{cc}
        \binom{r}{i}t^{r-i} & \mbox{ if } r\geq i \\
        0 & \mbox{otherwise}.
    \end{array}\right.$$ If $k$ has characteristic zero, then $\frac{\partial}{\partial t^i}=\frac{1}{i!}\left(\frac{\partial}{\partial t}\right)^i$. On the other hand, if $k$ has characteristic $p>0$ this does not make sense for $i=0\mod p.$ In this case if $i=jp^s$ for some $s\geq0$ with $j\neq0\mod p$, then $$\frac{\partial}{\partial t^i}=\frac{\partial}{\partial t^{jp^s}}=\frac{1}{j!}\left(\frac{\partial}{\partial t^{p^s}}\right)^j.$$ We will denote by $\partial_{p^s}$ the differential operator $\frac{\partial}{\partial t^{p^s}}.$
    \end{enumerate}
\end{example}

\section{Nilpotent derivations and  \texorpdfstring{$p$}{p}-bases}\label{nilpder}

This part will be devoted to nilpotent derivations, which are often encountered when studying actions of infinitesimal group schemes, as seen for example in \ref{alphap}. Some of the results appearing here might be known to experts, we have included their proof for lack of a reference. Proposition \ref{pbasis} and Corollary \ref{systemsol} will play an important role in the proof of Theorem \ref{mainthm}. 
A background reference for $p$-bases is \cite[V.\textsection13]{Bourbaki}.

For $K$ a field and $D$ a derivation on $K$, we will denote by $K^D$ the subfield of elements of $K$ annihilated by $D$, that is $$K^D:=\{x\in K\mid D(x)=0\}.$$
A derivation $D$ of $K$ is said to be nilpotent of nilpotency index $r$ if $D^r=0$ and $D^{r-1}\neq0$, where for any natural $n$ we denote by $D^n$ the composite of $D$ with itself iterated $n$ times.
We begin by recalling a fundamental result showing that nilpotent derivations on fields appear only in characteristic $p>0$ and that nilpotency indices are always $p$-powers.

\begin{thm}\label{smitsthm}
    Let $D$ be a non-zero nilpotent derivation of nilpotency index $r$ on a field $K$. Then $K$ has characteristic $p\neq0$ and $r=p^t$.
\end{thm}

\begin{proof}
   See \cite[Theorem 2]{Smits}.
\end{proof}

\begin{defi}[Order of a derivation]
    Let $D$ be a derivation on a field $K$. We say that $D$ has \textit{order} $r$ if it is nilpotent of index $r$.
\end{defi}

By Theorem \ref{smitsthm} we have that, if $p$ is a prime number and $t\geq1$, a derivation $D$ on a field $K$ has order $p^t$ if and only if $D^{p^{t-1}}\neq 0$ and $D^{p^t}=0$. From now on $K$ will be a field of characteristic $p>0$ such that $K/K^p$ is finite (this is the case for example if we take $K$ to be the function field of a variety over a perfect field).

\begin{lemma}\label{onepreimage}
    Let $D$ be a derivation on $K$ of order $p^n$. Then the field extension $K/K^D$ has order $p^n$. Moreover, there exists $t\in K$ such that $D(t)=1$ and $\Ima (D^i)=\ker (D^{p^{n}-i})$ for any $i=1,\dots, p^{n}$.
\end{lemma}

\begin{proof}
We have that $D$ is a nilpotent  $K^D$-linear map with one-dimensional kernel generated by $1$. Therefore there is a unique block in the normal Jordan form of $D$ of size $p^n$, computed with respect to a basis, which we can suppose that contains $1$. This implies that
 $\dim_{K^D}K=p^n$ and there exists $t$ such that $D(t)=1$. Moreover, since there is only one nilpotent Jordan block, it is clear that $\Ima (D^i) = \ker (D^{p^{n}-i})$ for any $i=1,\dots,p^n$.
\end{proof}

\begin{defi}[$p$-basis]
Let $K/L$ be a finite field extension such that $K^p\subseteq L$.
  A \textit{$p$-basis} of $K/L$ is a sequence $(t_1,\dots,t_n)\in K^n$ such that the monomials $t_1^{m_1}\dots t_n^{m_n}$ with $0\leq m_1,\dots,m_n\leq p-1$ form an $L$-basis of $K$.
\end{defi}

\begin{rmk}\label{canbasis}\leavevmode
\begin{enumerate}
    \item Notice that for any $p$-basis $(t_1,\dots,t_n)$ of $K/L$, a derivation $D$ in $\mathrm{Der}_L(K)$ is zero if and only if $D(t_i)=0$ for all $i=1,\dots,n$.
    \item A sequence $(t_1,\dots,t_n)$ is a $p$-basis of $K/L$ if and only if $\{dt_1,\dots,dt_n\}$ is a basis of the $K$-vector space of K\"ahler differentials $\Omega_{K/L}^1$ \cite[V.\textsection13, Theorem 1]{Bourbaki}. Consider the dual basis $\{\partial_1,\dots,\partial_n\}$, which gives a basis of the $K$-vector space of derivations $\mathrm{Der}_L(K)=\Hom_K(\Omega^1_{K/L},K).$ The $\partial_i$'s commute pairwise and satisfy $\partial_i^p=0$ for all $i=1,\dots,n.$ Moreover $\partial_i(t_j)=\delta_{ij}$, where $\delta_{ij}$ is the Kronecker delta.
\end{enumerate}
\end{rmk}

In the following, we construct special $p$-bases that can be obtained any time we have a generically free rational action of an infinitesimal commutative unipotent group scheme of height one on a variety.

\begin{prop}\label{pbasis}
       Let $D_1,\dots,D_n$ be 
    derivations on $K$ commuting pairwise. Set $K_0=K$ and $K_{j}=K^{D_1,\dots,D_{j}}$ for all $j=1,\dots,n.$ If $D_i$ has order $p$ on $K_{i-1}$ for all $i=1,\dots,n$ then
    \begin{enumerate}
    \item there exists a $p$-basis $(t_1,\dots,t_{n})$ of $K/K_n$ such that $D_i(t_i)=1$ and $D_i(t_j)=0$ for all $j>i$. Moreover,  $D_i(t_j)$ belongs to $K_j$ for all $i$ and $j$, and
\item   $\{D_1,\dots,D_n\}$ is a basis of $\mathrm{Der}_{K_n}(K)$.
\end{enumerate}
\end{prop}

\begin{proof}\leavevmode
\begin{enumerate}
    \item Consider the tower of extensions $$K_n\subseteq K_{n-1}\subseteq K_{n-2}\subseteq\dots\subseteq K_1\subseteq K.$$ By Lemma \ref{onepreimage}, for every $i=1,\dots,n$ the extension $K_i\subseteq K_{i-1}$ has degree $p$ and there exists $t_i\in K_{i-1}$ such that $D_i(t_i)=1$, so by degree reasons $K_{i-1}=K_i(t_i)$. Therefore the first statement follows. The second statement is a direct consequence of the first one together with the commutativity hypothesis. Indeed, for every $i,j$ and $h\leq j$ we have $$D_h(D_i(t_j))=D_i(D_h(t_j))=D_i(\delta_{hj})=0$$ where $\delta_{hj}$ is the Kronecker delta. Hence $D_i(t_j)$ belongs to $K_j$ as claimed.
    \item In particular, we see that $[K:K_n]=p^n$. As remarked above, $\mathrm{Der}_{K_n}(K)$ has then dimension $n$ over $K$, hence it is enough to show that $D_1,\dots,D_n$ are $K$-linearly independent. Suppose that they are not and take $a_1,\dots,a_n$ in $K$ such that $$a_1D_1+\dots+a_nD_n=0.$$ Let $i_0=\max\{i=1,\dots,n\mid a_i\neq0\}$. Then $$0=(a_1D_1+\dots+a_nD_n)_{|K_{i_0-1}}=a_{i_0}{D_{i_0}}_{|K_{i_0-1}}.$$ By assumption $D_{i_0}$ has order $p$ on $K_{i_0-1}$, so in particular it is different from zero. Thus $a_{i_0}=0$, which gives a contradiction.
\end{enumerate}
\end{proof}

We introduce some notation in order to prove the following Corollary that gives necessary and sufficient conditions for some systems of differential equations to have a solution. It will play a crucial role for the existence of the generically free actions of Theorem \ref{mainthm}. Let $D_1,\dots,D_m$ be differential operators on $K$ commuting pairwise and $a_1,\dots,a_m$ be elements of $K$ such that $$D_i(a_j)=D_j(a_i)$$ for all $i,j=1\dots,m.$ Consider moreover a polynomial $F\in\left( X_1,\dots,X_m\right) k[X_1,\dots,X_m]$ and write $$F=X_1Q_1+\dots+X_mQ_m.$$ We define the differential operator $$\widetilde{F}(a_1,\dots,a_m):=\sum_{i=1}^ma_iQ_i(D_1,\dots,D_m).$$
Notice that it does not depend on the choice of the $Q_i$'s since $D_i(a_j)=D_j(a_i)$ for every $i,j.$

\begin{corollary}\label{systemsol}
   Let $D_1,\dots,D_m$ be differential operators on $K$ as above and set $K_0=K$ and $K_j=K^{D_1,\dots,D_j}$ for any $j=1,\dots,m$. Suppose moreover that $D_i$ is a derivation of order $p^{l_i}$ on $K_{i-1}$ for any $i=1,\dots,m$ and that $$D_i^{p^{l_i}}=F_i(D_1,\dots,D_{i-1})$$ for some polynomial $F_i\in(X_1,\dots,X_{i-1})k[X_1,\dots,X_m].$ Then the system $$\left\{
        \begin{array}{c}
             D_1(x)=a_1 \\
             \vdots\\
             D_m(x)=a_m
        \end{array}\right.
    $$ admits a solution in $K$, which is unique modulo $K_m$, if and only if $$D_i^{p^{l_i}-1}(a_i)=\widetilde{F}_i(a_1,\dots,a_{i-1})$$ for every $i=1,\dots,m.$
\end{corollary}

\begin{proof}
    Suppose that there exists $z\in K$ solution of the above system: then $$D_i^{p^{l_i}-1}(a_i)=D_i^{p^{l_i}}(z)=F_i(D_1,\dots,D_{i-1})(z)=\widetilde{F}_i(a_1,\dots,a_{i-1})$$
    for every $i=1,\dots,m.$ For the other way around, notice that the uniqueness modulo $K_m$ of the solution is clear by the additivity of differential operators: if $x$ and $y$ are both solutions of the system, then $0=D_i(x)-D_i(y)=D_i(x-y)$ for all $i=1,\dots,m$, meaning that the two solutions differ by an element of $K_m$. Let us prove its existence by induction. Let $S_i$ be the system given by just the first $i$ lines for any $i=1,\dots,m$. 
    Let us show that if $S_i$ has a solution $x_i$, then $S_{i+1}$ has a solution. Any solution of $S_i$ is of the form $x_i+y_i$ with $y_i\in K_i$, therefore we wish to find such an element satisfying $$D_{i+1}(x_i+y_i)=a_{i+1}.$$ This equation is satisfied if and only if $$D_{i+1}(y_i)=a_{i+1}-D_{i+1}(x_i).$$ For every $j=1,\dots,i$ we have that $$D_j(D_{i+1}(x_i))=D_{i+1}(D_j(x_i))=D_{i+1}(a_j)=D_j(a_{i+1}),$$ that is $a_{i+1}-D_{i+1}(x_i)$ lies in $K_i$. Moreover,  $$D_{i+1}^{p^{l_{i+1}}-1}(a_{i+1}-D_{i+1}(x_i))=D_{i+1}^{p^{l_{i+1}}-1}(a_{i+1})-D_{i+1}^{p^{l_{i+1}}}(x_i)=$$$$\widetilde{F}_{i+1}(a_1,\dots,a_i)-D_{i+1}^{p^{l_{i+1}}}(x_i)=0$$ and thus by Lemma \ref{onepreimage} there exists the solution $y_i$ in $K_i$ we were looking for.
\end{proof}

\section{Generically free rational actions}\label{genfreeactions}

We begin this section with a useful criterion in order to determine when an action of an infinitesimal group scheme is generically free. An analogous criterion will be given later on for faithful actions of finite group schemes (see Proposition \ref{ifffaithful}). 

\begin{prop}\label{kerfrob}
Let $G$ be an infinitesimal $k$-group scheme and $X$ an irreducible $k$-scheme endowed with a $G$-action. Then:
    \begin{enumerate}
        \item the $G$-action is generically free if and only if the induced $\ker(F_G)$-action is generically free;
        \item if in addition $k$ is perfect and $G$ is commutative, the $G$-action  is generically free if and only if the induced action of $\soc(G)$ is generically free.
    \end{enumerate}
\end{prop}

\begin{proof} \leavevmode
    \begin{enumerate}
        \item Clearly if the $G$-action is generically free then also the induced $\ker(F_G)$-action is generically free. Suppose that the $G$-action on $X$ is not generically free. Let $\eta$ be the generic point of $X$ and $K=k(\eta)$. Then $\stab_G(\eta)$ is a non-trivial subgroup scheme of $G_K$ and thus $$ \ker(F_{G_K})\times_{G_K}\stab_G(\eta)= \ker(F_{G})_K\times_{G_K}\stab_G(\eta)\stackrel{\ref{stabsgrp}}{=}\stab_{\ker(F_G)}(\eta)$$ is non-trivial. 
    Henceforth the action of $\ker(F_G)$ on $X$ is not generically free.
        \item Clearly if the $G$-action is generically free then also the induced $\soc(G)$-action is generically free. For the other way around, by Proposition \ref{bschngenfree} and Lemma \ref{diag} we may suppose that $k=\overline{k}$. Then $G$ is trigonalizable.
        Suppose by contradiction that the $G$-action on $X$ is not generically free. Let $\eta$ be the generic point of $X$ and $K=k(\eta)$. Then $\stab_G(\eta)$ is a non-trivial subgroup scheme of $G_K$ and thus $$ \soc(\stab_G(\eta))\stackrel{\ref{soc1}}{=}\soc(G_K)\times_{G_K}\stab_G(\eta)\stackrel{\ref{diag}}{=}$$$$\soc(G)_K\times_{G_K}\stab_G(\eta)\stackrel{\ref{stabsgrp}}{=}\stab_{\soc(G)}(\eta)$$ is non-trivial by Lemma \ref{soc1}. Henceforth the action of $\soc(G)$ on $X$ is not generically free which gives a contradiction.
    \end{enumerate}
\end{proof}

Recall the following definition.

\begin{defi}[Solvable group scheme]
    A $k$-group scheme $G$ is said to be $k$-\textit{solvable} if it is affine and it admits a composition series with quotients isomorphic either to $\G_{a,k}$ or to $\G_{m,k}$ (see for example \cite[IV.\textsection4, Definition 3.1]{DG}).
\end{defi}

Remark that every $k$-solvable group scheme is smooth and connected.

\begin{prop}\label{solvablegrps}
    Let $G$ be a $k$-group scheme.
    \begin{enumerate}
        \item If $G$ is $k$-solvable, then $G$ is trigonalizable and its maximal unipotent $k$-subgroup scheme $G^u$ is $k$-solvable. Moreover $G$ is isomorphic as a $k$-scheme to $\G_{m,k}^{n-r}\times_k\G_{a,k}^r$ where $n=\dim(G)$ and $r=\dim(G^u)$.
        \item If $k$ is perfect and $G$ is trigonalizable, smooth and connected, then $G$ is $k$-solvable.
    \end{enumerate}
\end{prop}

\begin{proof}
    See for example \cite[IV.\textsection4, Proposition 3.4 and Corollary 3.8]{DG}.
\end{proof}

The following Proposition proves the existence part of Theorem \ref{mainthm} in the case of commutative trigonalizable group schemes of height one (see Remark \ref{theorem: case height 1}).

\begin{prop}\label{frobkeraction}
Let $\mathcal{G}$ be a $k$-solvable group scheme of dimension $n$, consider $G=\ker(F_\mathcal{G}^s\colon\mathcal{G}\rightarrow\mathcal{G})$ for some $s\geq1$ and let $X$ be a $k$-variety of dimension $\ell$. Then there exist generically free rational actions of $G$ on $X$ if and only if $n\leq \ell.$
\end{prop}

\begin{proof}
Suppose that there exists a generically free rational action of $G$ on $X$. Then, by Proposition \ref{dimlie}, $$n=\dim(\mathcal{G})=\dim_k(\Lie (G))\leq\dim(X)=\ell.$$ For the converse, let us start by proving that any variety $X$ of dimension $\ell$ admits a generically free rational action of $G$ if $n=\ell$.
By Proposition \ref{solvablegrps}, $G$ is a subscheme of $$\mathcal{G}\simeq\G_{m,k}^{n-r}\times_k\G_{a,k}^r$$ where $r=\dim(\mathcal{G}^u)$ and thus there is a natural generically free $G$-action on $\G_{m,k}^{n-r}\times_k\G_{a,k}^r$ by multiplication, since $G$ is a subgroup scheme of $\mathcal{G}$. We then have the $G$-torsor given by the Frobenius $$F^s\colon\G_{m,k}^{n-r}\times_k\G_{a,k}^r\rightarrow\G_{m,k}^{n-r}\times_k\G_{a,k}^r.$$ Let $K=k(X)$ and take any point $x\in\left(\G_{m,k}^{n-r}\times_k\G_{a,k}^r\right)\left(kK^{p^s}\right)$, $$x=(x_1,\dots,x_n)\colon\spec\left(kK^{p^s}\right)\rightarrow\G_{m,k}^{n-r}\times_k\G_{a,k}^r.$$ Then we have a $G$-torsor $$Y_x=\spec\left(kK^{p^s}[T_1,\dots,T_n]/(T_i^{p^s}-x_i)_{i=1,\dots,n}\right)\longrightarrow\spec(kK^{p^s})$$ given by the pull-back diagram  \begin{center} \begin{tikzcd} Y_x\arrow[r]\arrow[d]&\spec(kK^{p^s})\arrow[d,"x"]\\ \G_{m,k}^{n-r}\times_k\G_{a,k}^r\arrow[r,"F^s"]&\G_{m,k}^{n-r}\times_k\G_{a,k}^r.  \end{tikzcd}\end{center} Let $\{y_1,\dots,y_n\}$ be a $p$-basis for $K/kK^p$ and $x=(x_1,\dots,x_n)\in\left(\G_{m,k}^{n-r}\times_k\G_{a,k}^r\right)\left(kK^{p^s}\right)$ be the point with coordinates $x_i=y_i^{p^s}$ for $i=1,\dots,n$. Let us show that there is an isomorphism $$kK^{p^s}[T_1,\dots,T_n]/(T_i^{p^s}-x_i)_{i=1,\dots,n}\simeq K, T_i\mapsto y_i.$$ First of all, let us see that $kK^{p^s}[T_1,\dots,T_n]/(T_i^{p^s}-x_i)_{i=1,\dots,n}$ is a field. We can see this by induction on $n$: in fact, $$kK^{p^s}[T_1]/(T_1^{p^s}-x_1)$$ is a field since $T_1^{p^s}-x_1$ is irreducible in $kK^{p^s}[T_1]$ since $y_1\not\in kK^p.$ Without loss of generality we can then suppose by induction that $$L:=kK^{p^s}[T_1,\dots,T_{n-1}]/(T_i^{p^s}-x_i)_{i=1,\dots,n-1}$$ is a field and consider $$L[T_n]/(T_n^{p^s}-x_n).$$ The polynomial $T_n^{p^s}-x_n$ is irreducible in $L[T_n]$ since $y_n\not\in kK^p$ and thus the claim. Consider the morphism of rings \begin{align*} \psi\colon kK^{p^s}[T_1,\dots,T_n]/(T_i^{p^s}-x_i)_{i=1,\dots,n}&\rightarrow K\\ T_i&\mapsto y_i,\end{align*} then, since the objects are fields, it is an injection and since the two fields have the same degree over $kK^{p^s}$ then $\psi$ is an isomorphism, as wished. Therefore we constructed a $G$-torsor $Y_x=\spec(K)\rightarrow\spec(kK^{p^s})$, that is there exists a generically free rational action of $G$ on $X$, as claimed.
For the general case, consider $$H=\ker(F^s\colon\mathcal{G}\times_k\G_a^{\ell-n}\rightarrow\mathcal{G}\times_k\G_a^{\ell-n}).$$ By what we have just proved, there exists a generically free rational action of $H$ on $X$, since $$\dim(\mathcal{G}\times_k\G_a^{\ell-n})=\ell=\dim(X).$$ Notice that $G=\ker(F^s\colon\mathcal{G}\rightarrow\mathcal{G})$ is a $k$-subgroup scheme of $H$, indeed it is the kernel of the projection $$\pi_2\colon H\rightarrow\G_a^{\ell-n}.$$ As a consequence, there exists also a generically free rational action of $G$ on $X$, as wished.

\end{proof}

\begin{rmk}\label{theorem: case height 1}
If $k$ is perfect and $G$ is a commutative trigonalizable $k$-group scheme of height one, then $$G\simeq\prod_{i=1}^tW_{n_i}^1\times_k\mu_p^l=\ker\left(F\colon\prod_{i=1}^tW_{n_i}\times_k\G_m^l\rightarrow\prod_{i=1}^tW_{n_i}\times_k\G_m^l\right)$$ for some $t,l,n_i\geq1$ (see for example 
\cite[IV.\textsection2, 2.14]{DG}) and $\prod_{i=1}^tW_{n_i}\times_k\G_m^l$ is a $k$-solvable group scheme of dimension equal to $\dim_k(\Lie (G))$. Hence, Proposition \ref{frobkeraction} applies in this case with $n=\dim_k(\Lie (G))$.
\end{rmk}

 The following is an asymptotic result for the dimension of varieties endowed with generically free rational actions of infinitesimal trigonalizable group schemes. This result will be made more precise in the commutative case and over a perfect field with Theorem \ref{mainthm}. 
 
\begin{corollary}\label{asymptotic}
    For every infinitesimal trigonalizable $k$-group scheme $G$ there exists an integer $r>0$ such that for every variety $X$ of dimension $\geq r$ there exist generically free rational actions of $G$ on $X$.
\end{corollary}

\begin{proof}
Any trigonalizable $k$-group scheme $G$ has a closed immersion in the smooth $k$-algebraic group $T_n$ of trigonalizable matrices for some $n$. This $k$-group scheme is $k$-solvable. Moreover, if $G$ is infinitesimal, it is contained in the kernel of some iterated of the Frobenius of $T_n$. Therefore, by the previous Proposition, any variety of dimension greater or equal to the dimension of $T_n$ admits a generically free $G$-action.
\end{proof}

\begin{rmk}\label{nilpotentder}
    Notice that, as a consequence of the above Proposition \ref{frobkeraction}, we have that for every variety $X$ of dimension $n$ and for any $j\leq n$, there exists a nilpotent $k$-linear derivation $D$ on $K=k(X)$ of order $p^j$. Indeed, by loc. cit. there exists a generically free rational action of $W_j^1$ on $X$. This corresponds to a module algebra structure \begin{align*}
        v\colon k[T]/(T^{p^j})&\rightarrow \mathrm{Der}_k(K)\\
        T&\mapsto D
    \end{align*} where $k[T]/(T^{p^j})$ is the $k$-Hopf algebra of $\alpha_{p^j}$, the Cartier dual of $W_j^1$. Then $D^{p^j}=0$ and, by \cite[III.\textsection2, Corollary 2.7]{DG}, $D,D^p,\dots,D^{p^{j-1}}$ are $K$-linearly independent (indeed otherwise $v$ would have a non-trivial kernel), hence $D$ has order $j$.
\end{rmk}

\subsection{Proof of the main Theorem}

We begin this part with three technical results that are the building blocks for the construction of generically free rational actions done in the proof of Theorem \ref{mainthm}: the main idea of the proof is to show that for $G$ an infinitesimal commutative unipotent $k$-group scheme of height $n$, a generically free rational action of $G_{n-1}=\ker(F_G^{n-1})$ on a variety $X$ can be extended to a rational action of $G$. The main Theorem will then be proved by an inductive argument with base step given by Proposition \ref{frobkeraction}, which settles the case of commutative trigonalizable group schemes of height one. 

Let us remark that the Theorem is stated for $k$ a perfect field and the main reason for this is that, thanks to Proposition \ref{infwitt}, over perfect fields we have a good control of the structure of any infinitesimal commutative unipotent group scheme in terms of its ambient space given by finite Witt vectors. Nevertheless, the constructive proof of Theorem \ref{mainthm} works over an arbitrary field and can be used to build generically free rational actions of infinitesimal commutative unipotent group schemes, provided one has enough control on the structure of the group scheme they are considering. This also motivates us to keep the notation $kK^p$ that, when $k$ is perfect, coincides with $K^p$.

Lemma \ref{actiononquotient} tells us that if $G_{n-1}$ acts on $X$, then $G$ acts already on $X^{(p)}$. Lemma \ref{extendtopbasis} shows that to extend a rational action of $G_{n-1}$, it is enough to define it on a $p$-basis of the function field $K=k(X)/kK^p$. Lemma \ref{commder} shows that under certain commutativity assumptions, some commutators are indeed derivations on $K$.

\begin{lemma}\label{actiononquotient}
    Let $G$ be an infinitesimal $k$-group scheme of height $n$ and let us denote $G_i:=\ker(F_G^i)$ for all $i=1,\dots,n-1$. Any action of $G_{n-1}$ on a $k$-variety $X$ induces naturally an action of $G/G_i$ on $X^{(p^i)}$. Moreover, if the $G_{n-1}$-action on $X$ is faithful, the same holds true for the induced $G/G_i$-action on $X^{(p^i)}$.
\end{lemma}

\begin{proof}
    Let $G_{n-1}\times_kX\rightarrow X$ be a faithful action. Then we have a naturally induced faithful action $G_{n-1}^{(p^i)}\times_kX^{(p^i)}\rightarrow X^{(p^i)}$ obtained by base change (the proof is the same as that of Lemma \ref{bschngenfaithful}) and therefore also of $$G/G_i\simeq \Ima(F_G^i)\subseteq G_{n-1}^{(p^i)}$$ on $X^{(p^i)}.$
\end{proof}

\begin{rmk}\label{rmktecn}\leavevmode
\begin{itemize}
\item  In the above setting, the composite $$G\times_kX^{(p)}\rightarrow G/\ker(F_G)\times_kX^{(p)}\rightarrow X^{(p)}$$ provides us naturally with an action of $G$ on $X^{(p)}$, via $F_G\colon G\rightarrow G^{(p)}$.
    \item Algebraically, this means that if we have a module algebra structure $$k[G_{n-1}^\vee]\rightarrow \mathrm{End}_k(B)$$ this induces a module algebra structure $$v\colon k[G^\vee]\rightarrow\mathrm{End}_k(\Ima(F_B)).$$ Let $\eta\colon \Ima(F_B)\rightarrow\Hom_k\left(k[G^\vee],\Ima(F_B)\right)$ be the corresponding morphism of algebras and recall that $\eta(b)(a)=v(a)(b)$ for every $a\in k[G^\vee]$ and $b\in B$ (see  Remark \ref{initialrmk}). Explicitly we then have that for every $a\in k[G^\vee]$ and $\beta\in B^{(p)}$ it holds $$\eta(F_B(\beta))(a)=\left(F_{\Hom_{k}(A,B)}\circ\eta^{(p)}(\beta)\right)(a)=F_B\circ\eta^{(p)}(\beta)\circ V_A(a)=F_B(v^{(p)}(V_A(a))(\beta))$$ where the first equality holds by functoriality of the Frobenius and the second one by Lemma \ref{Frobofmaps}.
\end{itemize}
    
\end{rmk}

Let $k$ be perfect, $G$ be an infinitesimal commutative unipotent $k$-group scheme of height $n$ and $G_{n-1}=\ker(F_G^{n-1})$. In order to simplify the notation we denote by $G_{n-1}^\vee$ its dual, that is $G_{n-1}^\vee:=(G_{n-1})^\vee=\mathrm{coker}(V_{G^\vee}^{n-1})=G^\vee/\Ima(V_{G^\vee}^{n-1}).$
In Proposition \ref{structureicu}, we showed that there exists a structure of $k$-group scheme on $\mathcal{G}:=G_{n-1}^\vee\times_k\A^{r_n}_k$ where $r_n=\dim_k(\Lie (\Ima(V_{G^\vee}^{n-1})))$ such that $$0\rightarrow\G_a^{r_n}\rightarrow\mathcal{G}\rightarrow G_{n-1}^\vee\rightarrow0.$$ Moreover, $G^\vee$ embeds in $\mathcal{G}$ realizing the exact sequence $$0\rightarrow G^\vee\rightarrow\mathcal{G}\rightarrow\G_a^{r_n}\rightarrow0.$$
    At the level of algebras, this is rephrased by saying that $$k[\mathcal{G}]=k[G_{n-1}^\vee][T_{1},\dots,T_{r_n}]$$ can be endowed with a structure of $k$-Hopf algebra (coming from that of Witt vectors as explained in the proof of Proposition \ref{structureicu}) such that $$\Delta(T_{j})=T_{j}\otimes1+1\otimes T_{j}+R_{j}$$ where $R_{j}$ is an element of $k[G_{n-1}^\vee]\otimes_kk[G_{n-1}^\vee].$ Moreover,  $$k[G^\vee]=k[G_{n-1}^\vee][T_{1},\dots,T_{r_n}]/(P_1,\dots,P_{r_n})$$ where the polynomials $P_j$ are primitive elements of $k[G_{n-1}^\vee][T_{1},\dots,T_{r_n}]$ congruent to $T_{j}^{p^{l_{j}}}$ for some $l_{j}\geq1$ modulo the augmentation ideal of $k[G_{n-1}^\vee]$ for every $j=1,\dots,r_n$. 

For any $r\leq \dim_k(\Lie (\Ima(V_{G^\vee}^{n-1})))$, consider the commutative $k$-Hopf algebra $$k[G_{n-1}^\vee][T_1,\dots,T_r]$$ corresponding to $G_{n-1}^\vee\times_k\A^{r}_k$ with $k$-group scheme structure induced by that of $\mathcal{G}$. Consider the non-commutative $k$-algebra $k[G_{n-1}^\vee]\langle T_1,\dots,T_r\rangle$ where the variables $T_i$ don't commute neither among them nor with the commutative subalgebra $k[G_{n-1}^\vee]$. We endow $k[G_{n-1}^\vee]\langle T_1,\dots,T_r\rangle$ of a $k$-Hopf algebra structure (which extends that of $k[G_{n-1}^\vee]$) defined as follows: one first takes the non-commutative free algebra $\Gamma=k\langle T_{ij},T_1,\dots,T_r\rangle_{1\leq i\leq n-1,1\leq j\leq s}$ where $s$ is minimal such that $G_{n-1}^\vee\subseteq(W_{n-1})^s$ (notice that the $T_{ij}$'s are the variables needed to define $k[G_{n-1}^\vee]$ , while the $r$ additional variables $T_1,\dots,T_r$ will each play the role of the $n$th coordinate in the corresponding copy of Witt vectors). We define $\Delta:\Gamma\rightarrow\Gamma\otimes_k\Gamma$, sending each variable to the element of $\Gamma\otimes_k\Gamma$ given by the comultiplication of Witt vectors and then extending this map to a morphism of algebras. Finally we quotient $\Gamma$ by the two-sided ideal given by the commutators of the variables $T_{ij}$ for $1\leq i\leq n-1,1\leq j\leq s$ and by the two-sided ideal defining $k[G_{n-1}^\vee]$. In this way, $\Delta$ defines a comultiplication on $k[G_{n-1}^\vee]\langle T_1,\dots,T_r\rangle$. Notice that after taking the quotient $\Delta$ becomes coassociative (while a priori before it was not): indeed, since $\Delta$ extends the comultiplication of $k[G_{n-1}^\vee]$, we just have to check the coassociativity for $T_1,\dots,T_r$ and the property holds since $\Delta(T_{j})=T_{j}\otimes1+1\otimes T_{j}+R_{j}$ where $R_{j}$ is an element of $k[G_{n-1}^\vee]\otimes_kk[G_{n-1}^\vee]$ for $j=1,\dots,r$. 

Before going on we illustrate this construction in the case of Example \ref{concretexample}: consider $G=\alpha_p\times_k\alpha_{p^2}$ which has height $2$ and whose Cartier dual is $G^\vee=\alpha_p\times_kW_2^1$. We then have that $G_1^\vee=\alpha_p\times_k\alpha_p$ and in Example \ref{concretexample} we showed that we can write the Hopf algebra of $G^\vee$ as $$k[G^\vee]=k[G^\vee_1][T_1]/(T_1^p)=k[S_0,T_0]/(S_0^p,T_0^p)[T_1]/(T_1^p).$$ In this case we will then have $\Gamma=k\langle S_0,T_0,T_1\rangle$ and \begin{align*}
    \Delta\colon S_0&\mapsto S_0\otimes1+1\otimes S_0,\\
    T_0&\mapsto T_0\otimes1+1\otimes T_0,\\
    T_1&\mapsto T_1\otimes1+1\otimes T_1-\sum_{k=1}^{p-1}\frac{1}{p}\binom{p}{k}T_0^k\otimes T_0^{p-k}.
\end{align*}
To obtain a $k$-Hopf algebra structure on $k[G_1^\vee]\langle T_1\rangle$ we then quotient $\Gamma$ by the two sided ideal generated by $S_0T_0-T_0S_0,$ $ S_0^p$ and $T_0^p$.

In the above setting, we have the following results. 

\begin{lemma}\label{extendtopbasis}
    Let $X$ be a $k$-variety of dimension $s$ with function field $K$ and $p$-basis $(t_1,\dots,t_s)$ of $K/kK^p$. Then for any set $\{x_{ih}\mid i=1,\dots,r,h=1,\dots,s\}$ of elements of $K$ and any module algebra structure $$\Tilde{v}\colon k[G_{n-1}^\vee]\rightarrow \mathrm{Diff}_k(K)$$ there exists a unique module algebra structure $$v\colon k[G_{n-1}^\vee]\langle T_1,\dots,T_r\rangle\rightarrow \mathrm{Diff}_k(K)$$ extending $\Tilde{v}$ and such that $v(T_i)(t_h)=x_{ih}$ for every $i$ and $h.$
\end{lemma}

\begin{proof}
Let us begin with the existence. Since $v_{|k[G_{n-1}^\vee]}=\Tilde{v}$, it is enough to show that we can define $D_i=v(T_i)$ satisfying the property of compatibility with products (\ref{property}) and such that $D_i(t_h)=x_{ih}$ for every $i$ and $h.$ By Proposition \ref{structureicu} we have that $$\Delta(T_i)=T_i\otimes1+1\otimes T_i+\sum_j\alpha_{ij}\otimes\beta_{ij}$$ with $\alpha_{ij}$ and $\beta_{ij}$ lying in $k[G_{n-1}^\vee]$ for all $i,j.$ Therefore we need to define $D_i$ in such a way that $$D_i(fg)=D_i(f)g+fD_i(g)+\sum_jv(\alpha_{ij})(f)v(\beta_{ij})(g)$$ for all $f,g\in K$. Recall that for $(t_1,\dots,t_{s})$ to be a $p$-basis of $K/kK^p$ means that $$\{t_1^{m_1}\dots t_s^{m_s}\mid 0\leq m_1,\dots,m_s\leq p-1\}$$ is a basis of $K$ as $kK^p$-vector space. By assumption $G_{n-1}$ acts on the generic point $Y=\spec(K)$ of $X$ and thus, by Lemma \ref{actiononquotient}, $G$ acts on $Y^{(p)}=\spec(kK^p)$. Therefore, the differential operator $D_i:=v(T_i)$ is defined on $kK^p$ for every $i=1,\dots,r$. We then define $$D_i(a t_h)=D_i(a)t_h+a x_{ih}+\sum_jv(\alpha_{ij})(a) v(\beta_{ij})(t_h)$$ and
$$D_i(t_ht_l)=x_{ih}t_l+t_h x_{il}+\sum_jv(\alpha_{ij})(t_h) v(\beta_{ij})(t_l)$$ for every $a\in kK^p$ and $h\leq l=1,\dots,s$. Applying recursively the formula $$D_i(fg)=D_i(f)g+fD_i(g)+\sum_jv(\alpha_{ij})(f)v(\beta_{ij})(g)$$ we define $D_i$ on all the monomials of the form $at_1^{m_1}\dots t_s^{m_s}$ with $a\in kK^p$ and $0\leq m_1,\dots,m_s\leq p-1$ and extend it by linearity to every element of $K$.  The fact that $D_i$ is well defined is a consequence of the coassociativity and cocommutativity of the Hopf algebra structure on $k[G_{n-1}^\vee]\langle T_1,\dots,T_r\rangle$.
The uniqueness of the module algebra structure comes by construction.  
\end{proof}

Given two strings of natural numbers $I=(i_1,\dots,i_n)$ and $J=(j_1,\dots,j_n)$, we say that $I$ is smaller than $J$ with respect to the lexicographic order, and we write $I<_{LEX}J$, if there exists $k\in \{1,\dots,n\}$ such that $(i_1,\dots,i_{k-1})=(j_1,\dots,j_{k-1})$ and $i_k<j_k$.

\begin{lemma}\label{commder}
    Let $A:=k[G_{n-1}^\vee]\langle T_{n1},\dots,T_{nr_n}\rangle$ be as above. Moreover,  write $$k[G_j^\vee]=k[G_{j-1}^\vee][T_{j1},\dots,T_{jr_j}]/(P_{j1},\dots,P_{jr_j})$$ as in Remark \ref{algebraic pov} for every $j\leq n-1$. Let $$v\colon A\rightarrow \mathrm{End}_k(B)$$ be an $A$-module algebra structure on a $k$-algebra $B$ 
    and let $D_{jh}:=v(T_{jh})$ for every $j=1,\dots,n$ and $h=1,\dots,r_j$. It holds that for any $h=1,\dots,r_n$ and $(s,t)<_{LEX}(n,h)$, if $D_{nh}$ commutes with every element of $v\left(k[G_{s-1}]\right)$ then $D_{nh}D_{st}-D_{st}D_{nh}$ is a derivation.
\end{lemma}

\begin{proof} 
Recall that, by Proposition \ref{structureicu}, for every $j=1,\dots,n$ and $h=1,\dots,r_j$ $$\Delta(T_{jh})=T_{jh}\otimes1+1\otimes T_{jh}+\sum_q\alpha_{jh}^q\otimes\beta_{jh}^q$$ where $\alpha_{jh}^q$ and $\beta_{jh}^q$ lie in $k[G_{j-1}^\vee]$ for all $q.$ 
Now $$\Delta(T_{nh}T_{st})=\Delta(T_{nh})\Delta(T_{st})=$$ $$T_{nh}T_{st}\otimes1+1\otimes T_{nh}T_{st}+T_{nh}\otimes T_{st}+T_{st}\otimes T_{nh}+$$ $$\sum_{q}\alpha_{nh}^{q}T_{st}\otimes\beta_{nh}^{q}+\sum_{q}\alpha_{nh}^{q}\otimes\beta_{nh}^{q}T_{st}+\sum_{q'}T_{nh}\alpha_{st}^{q'}\otimes\beta_{st}^{q'}+\sum_{q'}\alpha_{st}^{q'}\otimes T_{nh}\beta_{st}^{q'}+\sum_{q,q'}\alpha_{nh}^{q}\alpha_{st}^{q'}\otimes\beta_{nh}^{q}\beta_{st}^{q'}$$ and $$\Delta(T_{st}T_{nh})=\Delta(T_{st})\Delta(T_{nh})=$$ $$T_{st}T_{nh}\otimes1+1\otimes T_{st}T_{nh}+T_{nh}\otimes T_{st}+T_{st}\otimes T_{nh}+$$ $$\sum_{q}T_{st}\alpha_{nh}^{q}\otimes\beta_{nh}^{q}+\sum_{q}\alpha_{nh}^{q}\otimes T_{st}\beta_{nh}^{q}+\sum_{q'}\alpha_{st}^{q'}T_{nh}\otimes\beta_{st}^{q'}+\sum_{q'}\alpha_{st}^{q'}\otimes \beta_{st}^{q'}T_{nh}+\sum_{q,q'}\alpha_{st}^{q'}\alpha_{nh}^{q}\otimes\beta_{st}^{q'}\beta_{nh}^{q}.$$ Therefore $$\Delta(T_{nh}T_{st}-T_{st}T_{nh})=$$ $$(T_{nh}T_{st}-T_{st}T_{nh})\otimes1+1\otimes(T_{nh}T_{st}-T_{st}T_{nh})+$$ $$\sum_{q}\alpha_{nh}^{q}T_{st}\otimes\beta_{nh}^{q}+\sum_{q}\alpha_{nh}^{q}\otimes\beta_{nh}^{q}T_{st}+\sum_{q'}T_{nh}\alpha_{st}^{q'}\otimes\beta_{st}^{q'}+\sum_{q'}\alpha_{st}^{q'}\otimes T_{nh}\beta_{st}^{q'}$$$$-\sum_{q}T_{st}\alpha_{nh}^{q}\otimes\beta_{nh}^{q}-\sum_{q}\alpha_{nh}^{q}\otimes T_{st}\beta_{nh}^{q}-\sum_{q'}\alpha_{st}^{q'}T_{nh}\otimes\beta_{st}^{q'}-\sum_{q'}\alpha_{st}^{q'}\otimes \beta_{st}^{q'}T_{nh}.$$  
If $(s,t)<_{LEX}(n,h)$, using the hypothesis that $D_{nh}$ commutes with every element of $v\left(k[G_{s-1}]\right)$ we obtain that $$(D_{nh}D_{st}-D_{st}D_{nh})(fg)=m\circ(v\otimes v\circ\Delta(T_{nh}T_{st}-T_{st}T_{nh}))(f\otimes g)=$$ $$m\circ(v\otimes v\circ(T_{nh}T_{st}-T_{st}T_{nh})\otimes1+1\otimes(T_{nh}T_{st}-T_{st}T_{nh}))(f\otimes g)$$ for every $f,g\in B$. Hence the statement.
\end{proof}

We give now an example, showing how to construct explicitly generically free rational actions of the $p^m$-torsion of a supersingular elliptic curve on any curve. The aim is that the understanding of this baby case will help in getting through the proof of Theorem \ref{mainthm}.

\begin{example}\label{ptorsion}
    Take the self-dual infinitesimal commutative unipotent $k$-group scheme $$G=\ker(F-V\colon W_n^n\rightarrow W_n^n)=\spec\left(k[T_1,\dots,T_n]/(T_1^p,T_2^p-T_1,\dots,T_n^p-T_{n-1})\right).$$ If $k$ is algebraically closed, and $n=2m$, $G$ is the $p^m$-torsion of any supersingular elliptic curve over $k$: for the case $n=2$ see \cite[Chapter II, 15.5]{Oort} while the general case can be deduced from \cite[Section 5]{deJongOort} together with \cite[Theorem 1.2]{Oortpdiv} and a proof can also be found in \cite[Corollary 3.5]{gouthier2026infinitesimal}. 
    Let $X$ be any curve over $k$ and $K=kK^p(t)$ be its function field. Since $G$ is self-dual, to give a rational $G$-action on $X$ is equivalent to giving a module algebra structure $$v\colon\spec\left(k[U_1,\dots,U_n]/(U_1^p,U_2^p-U_1,\dots,U_n^p-U_{n-1})\right)\rightarrow \mathrm{Diff}_k(K).$$ We know that there exist generically free rational actions of
    the Frobenius kernel $$\ker(F_G)=\spec(k[T_n]/(T_n^p))\simeq\alpha_p$$ on $X$ (either by Proposition \ref{frobkeraction} or by \cite[Lemma 3.6 and Lemma 5.3]{brion2022actions}). In particular, any such action corresponds to choosing a non-zero derivation $D_1$ on $K$ of order $p$ or, equivalently, to giving a module algebra structure $$v\colon\spec\left(k[U_1]/(U_1^p)\right)\rightarrow \mathrm{Diff}_k(K).$$ We want to show that any such action can be extended to a generically free rational action of $G$ on $X$. To do so we show that for any $i=2,\dots,n$ any generically free rational action of $\ker(F_G^{i-1})$ on $X$ extends to a generically free rational action of $\ker(F_G^i).$ Notice that $$\ker(F_G^i)=\spec\left(k[T_{n-i+1},\dots,T_n]/(T_{n-i+1}^p,T_{n-i+2}^p-T_{n-i+1},\dots,T_n^p-T_{n-1})\right)$$ 
    and that to give a rational action of $\ker(F_G^i)$ on $X$ is equivalent to defining a module algebra structure $$v\colon\spec\left(k[U_1,\dots,U_i]/(U_1^p,U_2^p-U_1,\dots,U_i^p-U_{i-1})\right)\rightarrow \mathrm{Diff}_k(K).$$ Suppose then that we have a generically free rational action of $\ker(F_G^{i-1})$ given by differential operators $D_1,\dots,D_{i-1}$ where $D_j=v(T_j)$ for every $j=1,\dots,i-1$. To extend it to a rational action of $\ker(F_G^i)$ is equivalent to defining a differential operator $D_i=v(T_i)$ such that: \begin{enumerate}
        \item $D_{i}$ satisfies the property of compatibility with products (\ref{property});
        \item $D_{i}$ commutes with $D_{j}$ for every $j=1,\dots,i-1$;
        \item $D_{i}^{p}=D_{i-1}$.
    \end{enumerate} 
    By Lemma \ref{actiononquotient}, $D_i$ is defined on $kK^p.$ In particular, $$D_i(\beta^p)=v(T_i)(\beta^p)=\left(v(V(T_i))(\beta)\right)^p=(D_{i-1}(\beta))^p$$ for every $\beta\in K.$ By Lemma \ref{extendtopbasis}, we then have that $D_i$ is defined using property $1$ provided we choose $x=D_i(t).$ Therefore, the first property is satisfied by definition and we need to show that there exists $x$ such that also properties $2$ and $3$ are satisfied. By Lemma \ref{commder} and the fact that $T_i^p-T_{i-1}$ is a primitive element, we have that $D_iD_j-D_jD_i$ and $
     D_i^{p}-D_{i-1}$ are derivations for every $j=1,\dots,i-1$. Applying Remark \ref{canbasis}, we obtain that $D_{i}$ commutes with $D_{j}$ for every $j=1,\dots,i-1$ and $D_{i}^{p}=D_{i-1}$ if and only if the system $$\left\{\begin{array}{cc}
        D_j(x)=D_iD_j(t)  & j=1,\dots,i-1  \\
        D_i^{p-1}(x)=D_{i-1}(t)  & 
     \end{array}\right.$$ admits a solution $x=D(t).$ Notice first of all that the system is well defined, that is $D_i$ is defined on $D_j(t)$. In fact, by Corollary \ref{onepreimage} we can suppose that $D_1(t)=1$, therefore $D_1D_j(t)=D_jD_1(t)=D_j(1)=0$, that is $D_j(t)$ belongs to $kK^p$, on which $D_i$ is defined.
     Let $a_j:=D_iD_j(t)$ for $j=1,\dots,i-1$.
     By induction, the set $\{D_1,\dots,D_{i-1}\}$ is an ordered set of pairwise commuting differential operators and such that $D_{j}$ is a derivation of order $p$ on the subfield $K^{D_1,\dots,D_{j-1}}$. Moreover,  $$D_{j}(a_{l})=D_{l}(a_{j})$$ for all $j,l=1,\dots,i-1$, indeed by induction $$D_{j}D_{l}(t)=D_{l}D_{j}(t)$$ and thus $$D_{j}(a_{l})=D_{j}D_{i}D_{l}(t)=D_{i}D_{j}D_{l}(t)=D_{i}D_{l}D_{j}(t)=D_{l}D_{i}D_{j}(t)=D_{l}(a_{j})$$ as wished (we used the fact that $D_i$ satisfies properties $2$ and $3$ on $kK^p$). Moreover,  $D_{j}^{p}=D_{j-1}.$
    By Corollary \ref{systemsol} we then know that a solution of the system $$S=\left\{\begin{array}{c}
         D_1(x)=a_1\\
         \vdots \\
          D_{i-1}(x)=a_{i-1}
    \end{array}\right.$$ exists if and only if $$D_{j}^{p-1}(a_{j})=a_{j-1}$$ for all $j=1,\dots,i-1.$ The relation indeed holds true, in fact $$D_j^{p-1}(a_j)=D_j^{p-1}D_iD_j(t)=D_iD_j^p(t)=D_iD_{j-1}(t)=a_{j-1}$$ where again we used the fact that $D_j(t)\in kK^p$ and that $D_i$ commutes with $D_1,\dots,D_{i-1}$ on $kK^p.$ We are left to find a solution of $S$ which satisfies also the last equation $$D_i^{p-1}(x)=D_{i-1}(t).$$ Let then $z$ be a solution of $S$: we are looking for another solution of $S$ of the form $x=z+y$ with $y\in K^{D_1,\dots,D_{i-1}}.$ Therefore $x$ is a solution of $$D_i^{p-1}(x)=D_{i-1}(t)$$ if and only if $$D_i^{p-1}(y)=D_{i-1}(t)-D_i^{p-1}(z).$$ Notice that the right hand side belongs to $K^{D_1,\dots,D_{i-1}}\subseteq kK^p$ on which $D_i$ is a derivation of order $p$. Indeed,  for every $j=1,\dots,i-1$ it holds that $$D_jD_i^{p-1}(z)=D_i^{p-1}D_j(z)=D_i^{p-1}D_iD_j(t)=D_i^pD_j(t)=D_{i-1}D_j(t)=D_jD_{i-1}(t)$$ as wished. Therefore, by Lemma \ref{onepreimage}, $y$ exists if and only if $D_i(D_{i-1}(t)-D_i^{p-1}(z))=0$ which is satisfied since $$D_i^p(z)=D_{i-1}(z)=D_iD_{i-1}(t).$$
    Notice that the action constructed is generically free since it extends the generically free action of $\soc(G)\simeq\alpha_p$ (see Proposition \ref{kerfrob}).
\end{example}

We are now ready to prove our result in full generality.

\begin{namedthm*}{Theorem \ref{mainthm}}
\textit{Let $k$ be a perfect field of characteristic $p>0$ and $G$ be an infinitesimal commutative unipotent $k$-group scheme with Lie algebra of dimension $s$. Then for every $k$-variety $X$ of dimension $\geq s$ there exist generically free rational actions of $G$ on $X.$ Moreover, for any $r\geq1$, any generically free rational action of $\ker(F_G^r)$ on $X$ can be extended to a generically free rational action of $G$ on $X$.}
\end{namedthm*}

\begin{proof}
We begin by proving that if $X$ is a $k$-variety of dimension $s=\dim_k(\Lie (G))$, then $X$ admits a generically free rational action of $G$. By Proposition \ref{frobkeraction} and Remark \ref{theorem: case height 1}, there exists a generically free rational action of $\ker(F_G)$ on $X$.
Consider the filtration $$G_1\subseteq G_2\subseteq\dots\subseteq G_{n-1}\subseteq G_n=G$$ where $G_i:=\ker(F_G^i)$ and $n$ is the height of $G$.

\vspace{0.5em}

\textit{\underline{First induction (on $i$)}}: 

To show that there exists a generically free rational action of $G$ on $X$, we will prove that for every $i=2,\dots,n$ any generically free rational action of $G_{i-1}$ on $X$ extends to a generically free rational action of $G_i$ on $X$. Moreover, we will consider any possible extension of the actions. As a consequence, the second part of the statement will be satisfied by construction. Let $K=k(X)$ be the function field of $X$.
By Proposition \ref{infinitesimalactions}, to give a rational action of $G_i$ on $X$ is equivalent to endowing $K$ of a $k[G_i^\vee]$-module algebra structure, where $G_i^\vee$ is the Cartier dual of $G_i$. By Proposition \ref{structureicu}, $$k[G_i^\vee]=k[G^\vee_{i-1}][T_{i1},\dots,T_{ir_i}]/(P_{i1},\dots,P_{ir_i})$$ where $$G^\vee_{i-1}=\mathrm{coker}(V_{G^\vee}^{i-1})=\left(\ker(F_G^{i-1})\right)^\vee,$$ $r_i=\dim_k\left(\Lie\left(H_i\right)\right)$ with $H_i=\Ima(V_{G^\vee}^{i-1})/\Ima( V_{G^\vee}^i)$, $P_{ij}=T_{ij}^{p^{m_{ij}}}-Q_{ij}$ are primitive elements of $k[G^\vee_{i-1}][T_{i1},\dots,T_{ir_i}]$, $Q_{ij}$ are polynomials with coefficients in the augmentation ideal of $k[G^\vee_{i-1}]$ and $$\Delta(T_{ij})=T_{ij}\otimes1+1\otimes T_{ij}+R_{ij}$$ where $R_{ij}$ is an element of $k[G^\vee_{i-1}]\otimes_kk[G^\vee_{i-1}]$ for every $j=1,\dots,r_i.$ 
We then want to show that a $k[G^\vee_{i-1}]$-module algebra structure on $K$ extends to a $k[G_i^\vee]$-module algebra structure on $K$. A $k[G^\vee_{i-1}]$-module structure on $K$ is given by a morphism of algebras
    \begin{align*}
        v\colon k[G^\vee_{i-1}]&\rightarrow \mathrm{Diff}_k(K)\end{align*} satisfying the property of compatibility with products (see Proposition \ref{infinitesimalactions}). If we want to extend $v$ to $$k[G_i^\vee]=k[G^\vee_{i-1}][T_{i1},\dots,T_{ir_i}]/(T_{i1}^{p^{m_{i1}}}-Q_{i1},\dots,T_{ir_i}^{p^{m_{ir_i}}}-Q_{ir_i})\rightarrow \mathrm{Diff}_k(K)$$ we need to define $v(T_{ij})=D_{ij}$ for every $j=1,\dots,r_i$ in such a way that the above map is a $k[G_i^\vee]$-module algebra structure on $K$, that is the following properties are satisfied for every $j=1,\dots,r_i$:
    \begin{enumerate}
        \item $D_{ij}$ satisfies the property of compatibility with products (\ref{property});
        \item $D_{ij}$ commutes with $D_{kl}$ for every $(k,l)<_{LEX}(i,j)$;
        \item $D_{ij}^{p^{m_{ij}}}=v(Q_{ij})$.
    \end{enumerate} 
    Notice that by \cite[IV.\textsection2, 2.14]{DG} $$\ker(F_G)\simeq\prod_{j\in I}W_{m_{1j}}^1$$ where $I$ is a finite set and $\sum_{j\in I}m_{1j}=s$ which is the dimension of $\Lie(G)$. Then $$\left(\ker(F_G)\right)^\vee\simeq\prod_{j\in I}\alpha_{p^{m_{1j}}}$$ and thus, by \cite[III.\textsection2, Corollary 2.7]{DG}, to give a generically free rational action of $\ker(F_G)$ on $X$ corresponds to giving a set of pairwise commuting derivations $\{D_{1j}\}_{j\in I}$ on $K$ and with $D_{1j}$ of order $p^{m_{1j}}$ for every $j\in I$, such that all the $p$-powers of these derivations are $K$-linearly independent. Let $\{E_1,\dots,E_s\}$ be the ordered set $\left\{D_{1j}^{p^{k_j}}\mid0\leq k_j<m_{1j},j\in I \right\}.$ This family satisfies the hypothesis of Proposition \ref{pbasis} and therefore there exists a $p$-basis $\{t_1,\dots,t_{s}\}$ of $K/kK^p$ such that $E_i(t_i)=1$ and $E_i(t_j)=0$ for all $j<i$ and $i=1,\dots,s.$   
    By Lemma \ref{actiononquotient}, the rational action of $G_i$ is defined on $X^{(p)}\stackrel{\sim}{\dashrightarrow} X/\ker(F_G)$. Notice that the rational isomorphism is a consequence of the fact that both the field extensions $K^{\ker(F_G)}\subseteq K$ and $kK^p\subseteq K$ have degree $p^s$ (in the first case because the rational action of $\ker(F_G)$ on $X$ is generically free and $\ker(F_G)$ has order $p^s$ and in the second case by Proposition \ref{Liu}) and moreover $kK^p\subseteq K^{\ker(F_G)}$. In particular $$D_{ij}(F_K(\beta))=v(T_{ij})(F_K(\beta))=F_K(v(V_{G_i^\vee}({T_{ij}}))(\beta))$$ for every $\beta\in kK^p$.
    By Lemma \ref{extendtopbasis}, for any $j=1,\dots,r_i$, we then have that $D_{ij}$ is defined using property $1$, provided we choose $x_h^{ij}=D_{ij}(t_h)$ for $h=1,\dots,s$. Therefore the first property is satisfied by definition. We will show that we can choose $x_h^{ij}$ for every $h$ and $j$ in such a way that also properties $2$ and $3$ are satisfied. 

    \vspace{0.5em}
    
    \textit{\underline{Second induction (on $j$)}}: 
    
    We show that if $D_{kl}$ is defined for all $(k,l)<_{LEX}(i,j)$ then we can define $D_{ij}.$ Recall that for the moment $D_{ij}$ is defined on $kK^p$ which is the function field of $X^{(p)}$, so we have that the rational action of $G_i$ is defined on $X^{(p)}\stackrel{\sim}{\dashrightarrow}X/\ker(F_G)$ and we want to extend it to a rational action on $X.$ 

    \vspace{0.5em}
    
    \textit{\underline{Third induction (on $h$)}}: 
    
    We will show that if $D_{ij}$ is defined on $kK^p(t_1,\dots,t_{h-1})$, then we can extend its definition to $kK^p(t_1,\dots,t_{h})$\footnote{Geometrically here we are taking a filtration of $\ker(F_G)$ with successive quotients isomorphic to $\alpha_p$ and considering subquotients of $X$. For example, in the case in which $\ker(F_G)\simeq W_s^1$ with generically free rational action on $X$ given by a derivation $D_1$ of order $p^s$, we have the tower $$K^{D_1}\subseteq K^{D_1^p}=K^{D_1}(t_1)\subseteq\dots\subseteq K^{D_1^{p^{s-1}}}=K^{D_1}(t_1,\dots,t_{s-1})\subseteq K=K^{D_1}(t_1,\dots,t_{s})$$ corresponding to $$X\dashrightarrow X/\alpha_p\dashrightarrow X/W_2^1\dashrightarrow\dots\dashrightarrow X/W_{s-1}^1\dashrightarrow X/W_s^1$$ and $$\alpha_p=\soc(\ker(F_G))\subseteq W_2^1=\ker(F_G)\times_k\ker(V_G^2)\subseteq\dots\subseteq W_{s-1}^1=\ker(F_G)\times_k\ker(V_G^{s-1})\subseteq W_s^1=\ker(F_G).$$ Notice that this phenomenon did not occur in Example \ref{ptorsion} since there the Frobenius kernel was just $\alpha_p$.}. The base step is satisfied since $D_{ij}$ is defined on $kK^p$. We will show that if $D_{ij}$ is defined on $kK^p(t_1,\dots,t_{h-1})$ then the system $$N_h=\left\{\begin{array}{lr}
     D_{kl}D_{ij}(t_h)=D_{ij}D_{kl}(t_h),    & (k,l)<_{LEX}(ij) \\
     D_{ij}^{p^{m_{ij}}}(t_h)=Q_{ij}(D_{k'l'})_{(k',l')<_{LEX}(i,j)}(t_h)
    \end{array}\right.$$ has a solution where the unknown is $x_h^{ij}=D_{ij}(t_h)$. Remark that, for the system to admit a solution is equivalent to having properties $2$ and $3$ satisfied on $kK^p(t_1,\dots,t_h)$. Indeed, by Lemma \ref{commder} and the fact that $P_{ij}=T_{ij}^{p^{m_{ij}}}-Q_{ij}$ is a primitive element, we have that $D_{kl}D_{ij}-D_{ij}D_{kl}$ and $
     D_{ij}^{p^{m_{ij}}}-Q_{ij}(D_{k'l'})_{(k',l')<_{LEX}(i,j)} $ are derivations. Applying then Remark \ref{canbasis}, we obtain that $D_{ij}$ commutes with $D_{kl}$ for every $(k,l)<_{LEX}(i,j)$ and that $D_{ij}^{p^{m_{ij}}}=v(Q_{ij})$ as claimed. Notice that, in particular, $x_{h'}^{kl}=D_{kl}(t_{h'})$ is a solution of the analogous system for every $(k,l,h')<_{LEX}(i,j,h)$ by the assumption that $D_{kl}$ is defined on $K$ for all $(k,l)<_{LEX}(i,j)$ and that $D_{ij}$ is defined on $kK^p(t_1,\dots,t_{h-1})$. We will first show that the system $$S_h=\biggl\{
     D_{kl}D_{ij}(t_h)=D_{ij}D_{kl}(t_h),    \quad (k,l)<_{LEX}(ij)
    \biggr.$$ obtained removing the last equation has a solution and then prove that there exists a solution of it which is also a solution of the last equation of the system $N_h$. Remark that, for the system $S_h$ to admit a solution is equivalent to having property $2$ satisfied on $kK^p(t_1,\dots,t_h)$. First of all, notice that the system $S_h$ is well defined, that is that $D_{ij}$ is defined on $D_{kl}(t_h)$ for $(k,l)<_{LEX}(ij)$: indeed $$E_iD_{kl}(t_h)=D_{kl}E_i(t_h)=0$$ for every $i\geq h$, and thus, by Proposition \ref{pbasis}, $D_{kl}(t_h)$ belongs to $kK^p(t_1,\dots,t_{h-1})$ on which $D_{ij}$ is defined. Let $a_{kl}:=D_{ij}D_{kl}(t_h)$, therefore we are looking for a solution of the system $$S_h=\biggl\{
     D_{kl}(x)=a_{kl},\quad(k,l)<_{LEX}(ij). 
   \biggr.$$
    By induction the set $\{D_{kl}\mid(k,l)<_{LEX}(i,j)\}$ is an ordered set of pairwise commuting differential operators and such that $D_{kl}$ is a derivation of order $p^{m_{kl}}$ on the subfield $$\{a\in K\mid D_{k'l'}(a)=0\quad\forall(k',l')<_{LEX}(k,l)\}$$ by Lemma \ref{actiononquotient}. Moreover,  $$D_{kl}(a_{k'l'})=D_{k'l'}(a_{kl})$$ for all $(k,l),(k',l')<_{LEX}(i,j)$, indeed by induction $$D_{kl}D_{k'l'}(t_h)=D_{k'l'}D_{kl}(t_h)$$ and thus $$D_{kl}(a_{k'l'})=D_{kl}D_{ij}D_{k'l'}(t_h)=D_{ij}D_{kl}D_{k'l'}(t_h)=$$$$D_{ij}D_{k'l'}D_{kl}(t_h)=D_{k'l'}D_{ij}D_{kl}(t_h)=D_{k'l'}(a_{kl})$$ as wished. Notice that we used the fact that $D_{kl}(t_h)$ lies in $kK^p(t_1,\dots,t_{h-1})$ and that, by induction, on this subfield $D_{ij}$ commutes with the previous (LEX-order wise) differential operators. In addition, $D_{kl}^{p^{m_{kl}}}=Q_{kl}(D_{k'l'})_{(k',l')<_{LEX}(k,l)}$ where $Q_{kl}$ is an element of $k[T_{k'l'}]_{(k',l')<_{LEX}(k,l)}$ with vanishing constant coefficient.
    By Corollary \ref{systemsol}, we then know that a solution of the system $S_h$ exists if and only if $$D_{kl}^{p^{m_{kl}}-1}(a_{kl})=\widetilde{Q}_{kl}(a_{k'l'})_{(k',l')<_{LEX}(k,l)}$$ for all $(k,l)<_{LEX}(i,j).$ Write the polynomial $Q_{kl}$ as $$Q_{kl}(T_{k'l'})_{(k',l')<_{LEX}(k,l)}=\sum_{(\alpha,\beta)<_{LEX}(k,l)}\rho_{\alpha\beta}(T_{k'l'})_{(k',l')<_{LEX}(k,l)}T_{\alpha\beta}.$$ Then $$\widetilde{Q}_{kl}(a_{k'l'})_{(k',l')<_{LEX}(k,l)}=\sum_{(\alpha,\beta)<_{LEX}(k,l)}\rho_{\alpha\beta}(D_{k'l'})_{(k',l')<_{LEX}(k,l)}a_{\alpha\beta}=$$$$\sum_{(\alpha,\beta)<_{LEX}(k,l)}\rho_{\alpha\beta}(D_{k'l'})_{(k',l')<_{LEX}(k,l)}D_{ij}D_{\alpha\beta}(t_h)=$$$$D_{ij}\sum_{(\alpha,\beta)<_{LEX}(k,l)}\rho_{\alpha\beta}(D_{k'l'})_{(k',l')<_{LEX}(k,l)}D_{\alpha\beta}(t_h)=D_{ij}Q_{kl}(D_{k'l'})_{(k',l')<_{LEX}(k,l)}(t_h)=$$$$D_{ij}D_{kl}^{p^{m_{kl}}}(t_h)=D_{kl}^{p^{m_{kl}}-1}D_{ij}D_{kl}(t_h)=D_{kl}^{p^{m_{kl}}-1}(a_{kl})$$ as needed, that is the system $S_h$ admits a solution.
    We are left to show that there exists a solution that satisfies also the equation $$D_{ij}^{p^{m_{ij}}-1}(z)=Q_{ij}(D_{k'l'})_{(k',l')<_{LEX}(i,j)}(t_h).$$ Notice that we are looking for a solution of the form $$x_h^{ij}=x+y$$ with $y$ in $$K^{\left\{D_{kl}\mid (k,l)<_{LEX}(i,j)\right\}}=\{a\in K\mid D_{kl}(a)=0\quad\forall(k,l)<_{    LEX}(i,j)\}$$ and $x$ a solution of $S_h$. Moreover,  notice that $x$ lies in $kK^p(t_1,\dots,t_{h-1}),$ indeed we remarked that if $S_h$ has a solution then $D_{ij}$ commutes with $D_{kl}$ for every $(k,l)<_{LEX}(i,j)$ on $kK^p(t_1,\dots,t_h)$, so in particular it commutes with $E_1,\dots,E_s$ which, we recall, are the $p$-powers of the derivations $D_{1j}$, $j\in I.$  Henceforth $$E_\eta(x)=E_\eta D_{ij}(t_h)=D_{ij}E_\eta(t_h)=0$$ for all $\eta\geq h.$ Therefore $x+y$ is a solution of the equation if and only if $$D_{ij}^{p^{m_{ij}}-1}(y)=Q_{ij}(D_{k'l'})_{(k',l')<_{LEX}(i,j)}(t_h)-D_{ij}^{p^{m_{ij}}-1}(x).$$ Let us show that the term on the right hand side lies in $K^{\left\{D_{kl}\mid (k,l)<_{LEX}(i,j)\right\}}.$
    Indeed,  for any $(k,l)<_{LEX}(i,j)$ it holds that $$D_{kl}D_{ij}^{p^{m_{ij}}-1}(x)=D_{ij}^{p^{m_{ij}}-1}D_{kl}(x)=D_{ij}^{p^{m_{ij}}-1}D_{ij}D_{kl}(t_h)=D_{ij}^{p^{m_{ij}}}D_{kl}(t_h)=$$$$Q_{ij}(D_{k'l'})_{(k',l')<_{LEX}(i,j)}D_{kl}(t_h)=D_{kl}Q_{ij}(D_{k'l'})_{(k',l')<_{LEX}(i,j)}(t_h)$$ where we used the fact that $x, D_{kl}(t_h)\in kK^p(t_1,\dots,t_{h-1})$ and that $x$ is a solution of the system $S_h$.
    Notice that $K^{\left\{D_{kl}\mid (k,l)<_{LEX}(i,j)\right\}}$ is a subfield of $kK^p=K^{\left\{D_{1j}\mid j\in I\right\}}$ and that $D_{ij}$ is a derivation  of order $p^{m_{ij}}$ on $K^{\left\{D_{kl}\mid (k,l)<_{LEX}(i,j)\right\}}.$ Therefore, by Lemma \ref{onepreimage}, $y$ exists if and only if $$D_{ij}\left(Q_{ij}(D_{k'l'})_{(k',l')<_{LEX}(i,j)}(t_h)-D_{ij}^{p^{m_{ij}}-1}(x)\right)=0$$ which is satisfied since $$D_{ij}^{p^{m_{ij}}}(x)=Q_{ij}(D_{k'l'})_{(k',l')<_{LEX}(i,j)}(x)=\sum_{(\alpha,\beta)<_{LEX}(i,j)}\rho_{\alpha\beta}(D_{k'l'})_{(k',l')<_{LEX}(i,j)}D_{\alpha\beta}(x)=$$$$\sum_{(\alpha,\beta)<_{LEX}(i,j)}\rho_{\alpha\beta}(D_{k'l'})_{(k',l')<_{LEX}(i,j)}D_{ij}D_{\alpha\beta}(t_h)=D_{ij}Q_{ij}(D_{k'l'})_{(k',l')<_{LEX}(i,j)}(t_h)$$ where the first equality is again a consequence of the fact that $x\in kK^p(t_1,\dots,t_{h-1})$. Notice that the action constructed is generically free since it extends the generically free action of $\soc(G)$ (see Proposition \ref{kerfrob}).
    
    For the general case of a variety $X$ of dimension $\dim(X)=\ell\geq s$, consider the infinitesimal commutative unipotent $k$-group scheme $G\times_k\alpha_p^{\ell-s}$ where $s=\dim_k(\Lie (G))$. Then $\dim_k(\Lie (G\times_k\alpha_p^{\ell-s}))=\ell$ and thus by what we just proved $X$ admits a generically free rational action of $G\times_k\alpha_p^{\ell-s}$. In particular, it admits a generically free rational action of its subgroup $G$. Moreover,  any generically free rational action of $\ker(F_G^r)$ on $X$ extends to a generically free rational action of $\ker(F_G^r)\times_k\alpha_p^{l-s}$ in the following way: consider the set of derivations $\{E_1,\dots,E_s\}$ defining the action of $\ker(F_G)$ on $K=L(t_1,\dots,t_s)$ where $L=k\left(X/\ker(F_G)\right)$ as described in the first induction and complete it to a basis $\{E_1,\dots,E_s,\partial_{s+1},\dots,\partial_\ell\}$ where the $\partial_i$'s are as in Remark \ref{canbasis}. One checks easily that the elements of this basis commute pairwise and that this implies that the $\partial_i$'s commute with all differential operators defining the rational action of $\ker(F_G^r)$ on $X$. By the case treated previously, the rational action of $\ker(F_G^r)\times_k\alpha_p^{l-s}$ extends to a generically free rational action of $G\times_k\alpha_p^{\ell-s}$ and thus, in particular, to a generically free rational action of $G$.
\end{proof}

\begin{rmk}\label{thm+Brion}
    Brion shows that for any $l,n\geq1$ there exist generically free rational actions of $\mu_{p^l}^n$ on any variety $X$ of dimension $\geq n$ \cite[Remark 3.8]{brion2022actions}.
    Putting together Brion's result and Theorem \ref{mainthm} one can prove that if $k$ is perfect and $G$ is an infinitesimal commutative trigonalizable $k$-group scheme with Lie algebra of dimension $s$, then for every $k$-variety $X$ of dimension $\geq s$ there exist generically free rational actions of $G$ on $X.$ By Theorem \ref{affinecommgrpschms}, $G\simeq G^u\times_k\prod_{i=1}^{s_2}\mu_{p^{n_i}}$ for some integers $n_i\geq1$ and we denote by $s_1$ and $s_2$ respectively the dimension of the Lie algebra of the unipotent part $G^u$ of $G$ and of its diagonalizable part. Briefly, one considers a set of derivations $\{E_1,\dots,E_{s_1}\}$ defining a generically free rational action of $\ker(F_{G^u})$ on $K=L(t_1,\dots,t_{s_1})$ where $L=k\left(X/\ker(F_{G^u})\right)$ as described in the first part of the proof of the Theorem and complete it to a $K$-linearly independent set $\{E_1,\dots,E_{s_1},t_{s_1+1}\partial_{s_1+1},\dots,t_{s_1+s_2}\partial_{s_1+s_2}\}$, with $\partial_i$'s as in Remark \ref{canbasis}. One checks easily that the elements of this basis commute pairwise and they thus define a generically free rational action of $\ker(F_G)$. Moreover, we can extend it as before to a generically free rational action of $\ker(F_G^r)$ (and so also of $G$) on $X$ for any $r\geq1$.
\end{rmk}

We conclude this section with two examples answering to some questions of Brion \cite{brion2022actions} and Fakhruddin \cite{Fakhruddin}.
Notice that if an infinitesimal commutative unipotent $k$-group scheme $G$ with $n$-dimensional Lie algebra can be embedded in a smooth connected $n$-dimensional algebraic group $\mathcal{G}$, then $G$ acts generically freely on it (by multiplication). Brion asked \cite[\textsection1]{brion2022actions} if there are examples of generically free rational actions on curves of infinitesimal commutative unipotent group schemes that are not subgroup schemes of a smooth connected one-dimensional algebraic group. Recall that if $\mathcal{G}$ is a smooth connected one-dimensional $k$-algebraic group, then either $\mathcal{G}$ is affine and $\mathcal{G}_{\overline{k}}\simeq\G_{m,\overline{k}}$ or $\mathcal{G}_{\overline{k}}\simeq\G_{a,\overline{k}}$ or $\mathcal{G}$ is an elliptic curve. 

\begin{example}\label{answerbrion}
Let $G$ be an infinitesimal commutative unipotent group scheme with one-dimensional Lie algebra, of order $>p^2$, with non-trivial Verschiebung and $p=0$. By Theorem \ref{mainthm} there exist generically free rational $G$-actions on any curve, since $G$ has one-dimension Lie algebra. Moreover such a $G$ cannot be contained in a smooth connected one-dimensional algebraic group. Indeed if it was the case, then it would be true also over $\overline{k}$ and $G_{\overline{k}}$ is not a subgroup neither of $\G_{m,\overline{k}}$ (since $G_{\overline{k}}$ is unipotent) nor of $\G_{a,\overline{k}}$ (since $V_{G_{\overline{k}}}\neq0$). Therefore, if $G$ is a subgroup scheme of a smooth connected one-dimensional algebraic group, then it is a subgroup of an elliptic curve $E$. In particular, since $p=0$, $G$ has to be contained in its $p$-torsion $E[p]$, but this cannot happen since $E[p]$ has order $p^2$ (see for example \cite{MumfordAbVar} page 137). A concrete example is given by the infinitesimal commutative unipotent $k$-group scheme $$G=\ker\left(F^2-V\colon W_3^3\rightarrow W_3^3\right)=\spec\left(k[T_0,T_1]/(T_0^p,T_1^{p^2}-T_0)\right)$$ which has one-dimensional Lie algebra, non-trivial Verschiebung, $p=V_GF_G=F_G^3=0$ and order $p^3$.
\end{example}

The following example goes in the direction of studying also non-commutative group schemes with generically free rational actions on curves.
Indeed, the question arises if all infinitesimal unipotent group schemes with one-dimensional Lie algebra are commutative (see both \cite[Remark 2.10]{Fakhruddin} and \cite[\textsection1]{brion2022actions}). The following example shows that it is not the case.

\begin{example}\label{noncommutative}
    Consider the infinitesimal unipotent non-commutative $k$-group scheme $G=\spec(A)$ where $$A=k[T_0,T_1]/\left(T_0^{p^n}, T_1^p-T_0\right)$$ with $n\geq2$ an integer and comultiplication given by $$\Delta(T_0)=T_0\otimes1+1\otimes T_0$$ and $$\Delta(T_1)=T_1\otimes1+1\otimes T_1+T_0^{p^{n-1}}\otimes T_0^{p^{n-2}}.$$ In this case $$A^\vee=k\langle U_0\dots,U_{n}\rangle/(U_0^p,\dots,U_n^p,U_iU_j-U_jU_i,U_nU_{n-1}-U_{n-1}U_n-U_0)_{i,j=0,\dots n,(i,j),(j,i)\neq(n,n-1)}$$ where $U_0(T_1)=1$ and $U_i(T_0^{p^{i-1}})=1$ and zero elsewhere. The Hopf algebra $A^\vee$ is non-commutative: the only non-commutative relation is given by $U_nU_{n-1}-U_{n-1}U_n=U_0$, while its comultiplication is defined on the $U_i$'s as for the Witt vectors, that is \begin{align*}
        U_0&\mapsto U_0\otimes1+1\otimes U_0,\\
        U_1&\mapsto U_1\otimes1+1\otimes U_1-\sum_{k=1}^{p-1}\frac{1}{p}\binom{p}{k}U_0^k\otimes U_0^{p-k},\\
        ...
    \end{align*} Notice that this makes sense since $U_0,\dots,U_{n-1}$ commute. Let $X/k$ be a curve and $K=k(X)=kK^p(t)$ be its function field, for $t$ a $p$-generator of $K$ over $kK^p$. A generically free rational action of $G$ on  $X$ is given by defining an $A^\vee$-module algebra structure on $K$ setting $v(U_i)=D_i=\partial_{p^i}$ for $i=0,\dots,n-1$ and $v(U_n)=D_n=\partial_{p^n}-t^{p^{n-1}}\partial_1.$ Notice that $\partial_{p^n}(t^{p^{n-1}})=0$ and thus $\partial_{p^n}$ commutes with $t^{p^{n-1}}\partial_1$. Therefore $$D_n^p=\partial_{p^n}^p-(t^{p^{n-1}}\partial_1)^p=\partial_{p^n}^p-t^{p^{n}}\partial_1^p=0$$ where for the second equality we used that also $\partial_1(t^{p^{n-1}})=0$ and for the last that $\partial_{p^n}^p=\partial_1^p=0.$ Of course this rational action can be extended to a generically free rational action of $G$ on any variety of positive dimension: if $X$ has dimension $s\geq0$ then $K=k(X)=kK^p(t=t_1,t_2,\dots,t_s)$ and a generically free action of $G$ on $X$ is given by the same set of differential operators as in the previous case of a curve. These examples arise as closed subgroup schemes of non-commutative extensions of $\G_a$ by itself (see \cite[II.\textsection3, 4]{DG}) and there are many of them. When $k$ has characteristic $2$ and $n=3$, $G$ is a subgroup scheme of $\mathrm{PGL}_{2,k}$ (see \cite[Theorem 1.1]{gouthier2024unexpectedsubgroupschemespgl2k}).
\end{example}

\section{Faithful rational actions}\label{faithactions}

This last section is devoted to Dolgachev's conjecture revisited for infinitesimal group schemes and more generally to studying faithful rational actions of infinitesimal group schemes. Dolgachev made the following conjecture for the Cremona group over a field of positive characteristic.

\begin{namedthm*}{Conjecture}
If $k$ is a field of characteristic $p>0$, the Cremona group $\mathrm{Cr}_n(k)$ does not contain elements of order $p^s$ for $s>n$
\cite[Conjecture 37]{Dolgachev}.
\end{namedthm*}

The conjecture is true for $n=1$ since $\mathrm{PGL}_2(k)\simeq \Aut_k(k(t))$ does not contain elements of order $p^2$ if $char(k)=p>0$. Moreover,  it was proven for $n=2$ \cite{Dolgachev1}. The conjecture can be rephrased in the following way: if there exists a faithful rational action of a finite commutative $p$-group $G$ on $\mathbb{P}^n_k$ then $p^n_G=0$, where $p_G$ is the multiplication by $p$ morphism on $G$. Indeed there is a natural correspondence between faithful actions of a finite group $G$ on $k(t_1,\dots,t_n)$ and faithful rational actions of the corresponding constant group scheme on $\Pj^n_k$. In fact, an action $G\times k(t_1,\dots,t_n)\rightarrow k(t_1,\dots,t_n)$ can be extended naturally to a $k[G]$-module algebra structure
    $k[G]\rightarrow \End_k(k(t_1,\dots,t_n))$,
where $k[G]$ is the group algebra over $G$ with its Hopf algebra structure (see \cite[Chapter 2.2]{Waterhouse}) and this gives a faithful rational action of the constant group scheme $G$ on $\Pj^n_k$. The analogous of Dolgachev's conjecture in our context is given by Proposition \ref{actionkerF}, that we will now prove after a couple of preliminary results.

\begin{prop}\label{ifffaithful}
Let $G$ be a finite $k$-group scheme and $X$ a $k$-scheme endowed with a $G$-action. The action is faithful if and only if the induced action of $\soc(G)$ is faithful.
\end{prop}

\begin{proof} 
The $G$-action is faithful if and only if the centralizer $C_G(X)$ is trivial. By Lemma \ref{soc1}, since $C_G(X)$ is a normal $k$-subgroup scheme of $G$, the centralizer is trivial if and only if $\soc(G)\times_GC_G(X)=C_{\soc(G)}(X)$ is trivial, that is if and only if the induced $\soc(G)$-action is faithful.
\end{proof}

The following result generalizes \cite[Lemma 5.3]{brion2022actions}.

\begin{corollary}\label{freeifffaith}
Let $G$ be an infinitesimal commutative $k$-group scheme, acting rationally on a $k$-variety $X$. Then the rational $G$-action is generically free if and only if it is faithful and the induced action of $\soc((G_{\overline{k}})^u)$ is generically free, where $(G_{\overline{k}})^u$ is the maximal unipotent subgroup scheme contained in $G_{\overline{k}}$. In particular, if $\soc((G_{\overline{k}})^u)\subseteq\alpha_p$, then the $G$-action is faithful if and only if it is generically free.
\end{corollary}

\begin{proof}
 By Lemma \ref{bschngenfaithful} and Proposition \ref{bschngenfree}, we may suppose $k=\overline{k}$. The \textit{only if} part is clear. We prove the other implication. We first prove the case $G$ diagonalizable. In this case, we have to prove that if the $G$-action is faithful then it is generically free.  By the anti-equivalence of categories between diagonalizable group schemes and abelian groups, if the stabilizer of the generic point $\spec(K)$ of $X$ is not trivial over $K$, then it comes from a non-trivial subgroup of $G$ over $k$, which then acts trivially, meaning that the action is not faithful. 

Now we pass to the general case. Since $k=\overline{k}$ then $G$ is isomorphic to $G^u\times_kG^d$, where $G^d$ is diagonalizable. Since the $\soc(G^u)$-action is generically free, the stabilizer at the generic point $\spec(K)$ should be contained in $G_K^d$, but this is not possible since the $G^d$-action is generically free by the diagonalizable case.
For the last sentence we observe that if $\soc(G^u)$ is a subgroup scheme of $\alpha_p$ and the $G$-action is faithful, then the $\soc(G^u)$-action is generically free. So we can apply the first part of the corollary. 
\end{proof}

\begin{rmk}\label{wittvectors}
The above Corollary applies, for instance, to any infinitesimal subgroup scheme $G$ of $W_n$, for some $n$: indeed in this case $\soc(G)=\alpha_p$ (see Example \ref{exasoc}).
\end{rmk}

In the following corollary we essentially get the second part of \cite[Lemma 3.7]{brion2022actions}

\begin{corollary}
    Let $G$ be an infinitesimal $k$-group scheme acting faithfully on a $k$-variety of dimension $r$. If $H$ is a normal $k$-subgroup scheme of $G$ of multiplicative type such that $\dim_k(\Lie (H))=r$, then $G$ is of multiplicative type and $\dim_k(\Lie (G))=r$.
\end{corollary}

\begin{proof}
By \cite[IV,\S 1, Corollary 4.4]{DG} we have that $H$ is central in $G$.
Now we can suppose that $k$ is algebraically closed, then $H$ is diagonalizable. If $G$ is not diagonalizable then $G$ contains a $k$-subgroup scheme isomorphic to $\alpha_p$ (\cite[IV \S 3, Lemma 3.7]{DG}). Then $H'=H\times_k \alpha_p$ is contained in $G$ since $H$ is central. Now $H'$ is commutative, $\soc(H')=\ker F_H\times_k\alpha_p$ and $\soc((H')^u)=\alpha_p$. Therefore, by Corollary \ref{freeifffaith}, the action of $H'$ is generically free, but this is impossible by Proposition \ref{dimlie} since $\dim_k(\Lie (H'))>r$. 
So $G$ is diagonalizable, its action is generically free (again by Corollary \ref{freeifffaith}), and $\dim_k(\Lie (G))$ can not be bigger than $r$.
\end{proof}

\begin{example}
    The condition on the normality of $H$ is crucial. Consider for example the $k$-group scheme $G=\alpha_p\rtimes \mu_p$ (where the action of $\mu_p$ on $\alpha_p$ is given by multiplication on the left) and the action on the affine line $G\times_k\A^1_k\rightarrow\A^1_k$ given by $(a,b)\cdot x\mapsto ax+b$. The $G$-action is faithful but not generically free (see Proposition \ref{dimlie}). The stabilizer of the generic point $\eta$ is $$\stab_G(\eta)=\spec\left(k(x)[T,1/T,S]/(xT+S-x,T^p-1,S^p)\right)$$ with comultiplication given by $\Delta(T)=T\otimes T$ and $\Delta(S)=S\otimes1+T\otimes S$. This is also a counterexample to Corollary \ref{freeifffaith} in the non-commutative case. Indeed $\soc(G)=\alpha_p$ and $\alpha_p$ acts freely. It is then necessary, in the non-commutative case, to look at the action of $\ker(F_G)$, as seen in the first statement of Proposition \ref{kerfrob}. 
\end{example} 

\begin{namedthm*}{Proposition \ref{actionkerF}}
    \textit{Let $G$ be an algebraic $k$-group scheme with commutative Frobenius kernel and $X$ be a $k$-variety of dimension $n$. If there exists a faithful rational $G$-action on $X$, then $s=dim_k(\Lie (\ker(F_{G})^m))\leq n$ and $V_{\ker(F_{G})^u}^{n-s}=0$, where $\ker(F_G)^m$ is the maximal $k$-subgroup scheme of multiplicative type of $\ker(F_G)$ and $\ker(F_G)^u:=\ker(F_G)/\ker(F_G)^m$.} 
\end{namedthm*}

\begin{proof}
We may suppose that $G$ is infinitesimal of height one and that $k$ is algebraically closed. Then $$G\simeq G^u\times_kG^m=\prod_{i\in I}W_{n_i}^1\times_k\mu_p^s.$$ Clearly $s=\dim_k(\Lie (G^m))\leq n$ since a faithful rational $\mu_p^s$-action is generically free. Let $l=\max_{i\in I}\{n_i\}$. By Corollary \ref{freeifffaith} the induced faithful rational action of $W_l^1\times_k\mu_p^s$ on $X$ is generically free. Hence $l+s\leq n$ and thus $V^{n-s}_{G^u}=0$.
\end{proof}

Notice that if $k$ is a perfect field, then the above Proposition tells us that if $G$ is an infinitesimal commutative trigonalizable $k$-group scheme such that there exists a faithful rational $G$-action on a $k$-variety of dimension $n$, then $\ker(F_G)^u\subseteq\left(W_{n-s}^1\right)^l$ for some $l\geq1$ where $s=\dim_k(\Lie(G^d))$. In particular, if there exists a faithful rational $G$-action on a curve, then $\ker(F_G)^u\subseteq\alpha_p^l$ for some $l\geq1$. 

The converse implication of Proposition \ref{actionkerF} does not always hold true. In the diagonalizable case, these actions are well understood and the converse statement is known.
Notice that by Remark \ref{thm+Brion} the converse of Proposition \ref{actionkerF} holds as well, over a perfect field, for infinitesimal commutative trigonalizable $k$-group schemes with Lie algebra of dimension upper bounded by the dimension of $X$. In particular, if $s=\dim_k(\Lie (\ker(F_{G})^d))$ and  $\dim_k(\Lie (G))\leq n$, then $V_{\ker(F_{G})^u}^{n-s}=0$. We will now give a counterexample to the converse implication of Proposition \ref{actionkerF}: we exhibit an infinitesimal commutative unipotent $k$-group scheme $G$ such that $V_{\ker(F_G)}=0$ but for which there are no faithful rational $G$-actions on any curve. We then keep investigating other cases in which the converse of Proposition \ref{actionkerF} holds.

\begin{example}\label{counterexample}
 Consider the $k$-subgroup scheme $G$ of $W_2\times_kW_2$ represented by the Hopf algebra $$k[T_0,T_1,U_0,U_1]/(T_0^p,T_1^p-T_0,U_0^p,U_1^p-U_0).$$ The $k$-group scheme $G$ is self-dual and, if $k$ is algebraically closed, $G\simeq E[p]\times_kE[p]$ for $E$ a supersingular elliptic curve over $k$ (that is $G$ is the $p$-torsion of a superspecial abelian surface). Moreover, $\dim_k(\Lie(G))=2$ and ${V_{\ker(F_G)}}=0$. Therefore, by Proposition \ref{dimlie}, we know that there is no generically free rational $G$-action on any curve. Let us show that moreover there is no faithful rational $G$-action on any curve either. Let $X$ be a curve and $K=k(X)$ be its function field. Suppose that there exists a faithful rational $G$-action on $X$ defined by the module algebra structure $$v:k[T_0,T_1,U_0,U_1]/(T_0^p,T_1^p-T_0,U_0^p,U_1^p-U_0)\rightarrow \mathrm{Diff}_k(K).$$ The differential operator $v(T_0)$ is a derivation on $K$ of order $p$, thus, by Lemma \ref{onepreimage}, there exists $x\in K$ such that $v(T_0)(x)=1$. Then, $v(T_0)=\partial_x$, the only $k$-linear derivation on $K=kK^p(x)$ such that $\partial_x(x)=1$. As a consequence, $v(U_0)=f_1\partial_x$ for some $f_1\in K$ since $\mathrm{Der}_k(K)$ is one-dimensional over $K$. Moreover, $f_1$ lies in $kK^p$ since the $v(T_0)$ and $v(U_0)$ commute, and is non-constant since the action is faithful and thus $v(T_0)$ and $v(U_0)$ must be $k$-linearly independent. Now, $v(T_1)(x)=x_1$ for some $x_1\in kK^p$, since $v(T_1)$ commutes with $\partial_x$. Moreover $v(T_1)_{|kK^p}$ is a derivation of order $p$ and $v(T_1)(x^p)=(v(T_0)(x))^p=1$ (see Remark \ref{rmktecn}). The differential operator $v(T_1)$ commutes also with $v(U_0)=f_1\partial_x$, hence $$v(T_1)(f_1)=f_1\partial_x(x_1)=0$$ that is $f_1$ must lie in $kK^{p^2}$. Moreover, $x_1$ is such that $$  v(T_1)^{p-1}(x_1)=v(T_1^p)(x)=v(T_0)(x)=1.$$ Consider now $v(U_1)$: as before, $v(U_1)(x)=x_2$ for some $x_2\in kK^p$ because of the commutativity with $\partial_x$, $v(U_1)_{|kK^p}$ is a derivation of order $p$ and by Remark \ref{rmktecn} we have $v(U_1)(x^p)=(v(U_0)(x))^p=f_1^p$. Hence $v(U_1)_{|kK^p}=f_1^pv(T_1)_{|kK^p}$. The differential operator $v(U_1)$ commutes also with $v(T_1)$, thus $$v(T_1)(x_2)=v(T_1)v(U_1)(x)=v(U_1)v(T_1)(x)=v(U_1)(x_1)=f_1^pv(T_1)(x_1)=v(T_1)(f_1^px_1).$$ Finally, $$f_1=v(U_0)(x)=v(U_1^p)(x)=(f_1^p)^{p-1}v(T_1)^{p-1}(x_2)=$$$$(f_1^p)^{p-1}v(T_1)^{p-1}(f_1^px_1)=f_1^{p^2}v(T_1)^{p-1}(x_1)=f_1^{p^2}$$ and this condition contradicts the fact that $f_1$ had to be non constant. Therefore there is no faithful rational $G$-action on any curve. Notice that we nevertheless showed that there exist faithful rational actions on any curve of the subgroup scheme $H$ of $G$ represented by the Hopf subalgebra $$k[T_0,T_1,U_1]/(T_0^p,T_1^p-T_0,U_1^p),$$ that is $H\simeq E[p]\times_k\alpha_p$ over $k=\overline{k}$. Actually, this kind of behaviour takes always place as shown in the following result.
\end{example}

The following Proposition generalizes a result of Brion \cite[Lemma 3.6]{brion2022actions} stating that every variety of positive dimension admits a faithful rational action of $\alpha_p^l$ for any $l\geq1$. 

\begin{prop}\label{faithrationalactions}
      Let $k$ be perfect, $G$ be an infinitesimal commutative unipotent $k$-group scheme and $X$ be a $k$-variety of dimension $n$. If $\dim_k(\Lie (G))\leq n$, then for every $\ell\geq0$ there exists a faithful rational action of $G\times_k\ker(F_G)^\ell$ on $X$.
\end{prop}

\begin{proof}
   Let $s=\dim_k(\Lie (G))$ and $K=k(X)$ be the function field of $X$. Then $\ker(F_G)$ corresponds to a certain Young diagram $(m_1,\dots,m_h)$ for some $h\geq 1$ (we refer the reader to \ref{sectionyoung}), $\sum_{i=1}^hm_{i}=s$ and $\ker(F_G)^\vee\simeq\prod_{i=1}^h\alpha_{p^{m_{i}}}.$ We know (Proposition \ref{frobkeraction}) that there exist generically free rational actions of $\ker(F_G)$ on $X$. By \cite[III.\textsection2, Corollary 2.7]{DG}, to give such a rational action corresponds to giving a set of derivations $\{D_1,\dots,D_h\}$ on $K$ commuting pairwise, with $D_{i}$ of order $p^{m_{i}}$ for every $i=1,\dots,h$ and such that all the $p$-powers of these derivations are $K$-linearly independent. Moreover, by Theorem \ref{mainthm} this action can be extended to a generically free rational $G$-action on $X$. Let $L=K^G$ be the function field of $X/G$. Take non-constant $k$-linearly independent elements $$\{f_{i1},\dots,f_{ih}\mid i=1,\dots,\ell\}$$ in $L$. Then $\{f_{i1}D_{1},\dots,f_{ih}D_{h}\}$ is still a set of derivations defining a generically free rational action of $\ker(F_G)$ on $X$. Consider then the induced rational action of $G\times_k\ker(F_G)^\ell$ on $X$ and notice that it is faithful by Proposition \ref{ifffaithful}. Indeed the rational action of $\ker(F_{G\times_k\ker(F_G)^\ell})=\ker(F_G)^{\ell+1}$ is given by the set of derivations $$\{D_1,\dots,D_h,f_{i1}D_{1},\dots,f_{ih}D_{h}\mid i=1,\dots,l\}$$ whose $p$-powers are $k$-linearly independent and thus it is faithful.
\end{proof}

\begin{rmk}\label{genfreermk} 
Notice that as a direct consequence we have that for any infinitesimal commutative unipotent $k$-group scheme $G$ of height one, there exist faithful rational $G^\ell$-actions on any $k$-variety of dimension $\geq\dim_k(\Lie(G))$ for any $\ell\geq1$.
In the proof we actually prove something more. Indeed we construct a faithful rational action of $G^\ell$ such that the induced action of any copy of $G$ is generically free. 
\end{rmk}

\begin{prop}\label{suffcondfaithact}
    Let $k$ be perfect, $G$ be an infinitesimal commutative unipotent $k$-group scheme and $X$ be a $k$-variety of dimension $n$. If $V_G^n=0$ then there exists a faithful rational $G$-action on $X$.
\end{prop}

\begin{proof}
We begin by recalling that, by Proposition \ref{infwitt}, $G$ can be embedded in $(W_n^m)^r$ for some $m,r\geq1$.
It is then enough to prove that there exists a faithful rational action of $(W_n^m)^r$ on $X$ for any $m,r\geq1$. By Proposition \ref{frobkeraction}, there exist generically free rational actions of $W_n^m$ on $X$ for any $m\geq1$ and, by Proposition \ref{infinitesimalactions}, to give such an action corresponds to giving a set of differential operators $\{D_0,\dots,D_{m-1}\}$ on the function field $k(X)$ of $X$ commuting pairwise, with $D_{i}$ of order $p^{i}$ and $p^n$-nilpotent ($D_i^{p^{n-1}}\neq0$) for every $i=0,\dots,m-1.$ Let $L$ be the function field of $X/W_n^m$. Take $k$-linearly independent elements $\{f_1,\dots,f_r\}$ in $L$. Then $\left\{f_iD_0,f_i^pD_1,\dots,f_i^{p^{m-1}}D_{m-1}\mid i=1,\dots,r\right\}$ gives a faithful rational action of $(W_n^m)^r$ on $X$. Indeed, since we took the $f_i's$ in $L$, these differential operators all commute pairwise and are moreover $p^n$-nilpotent. In addition, by the weighted homogeneity of Witt vectors, they respect the property of compatibility with products. Finally, the action is faithful because the action of the Frobenius kernel is faithful, since we chose $f_1,\dots,f_r$ linearly independent over $k$.
\end{proof}

Recall from \ref{sectionyoung} that if we take $G_1,\dots,G_l$ commutative unipotent $k$-group schemes of height one, there exists a smallest commutative unipotent $k$-group scheme $G$ of height one containing all of them. Precisely, $G$ corresponds to the smallest Young diagram containing all the Young diagrams $\tau(G_i)$ for all $i$. Explicitly, if $\tau(G_i)=(n_{1i},\dots,n_{s_ii})$ for some $s_i\geq 1$ and for $i=1,\dots,l$ then $\tau(G)=(n_1,\dots,n_s)$ where $s=\max\{s_1,\dots,s_l\}$ and $n_j=\max\{n_{j1},\dots,n_{jl}\}$ for every $j=1,\dots,s$. For example, if we take $G_1=W_3^1\times_k\alpha_p$ and $G_2=W_2^1\times_kW_2^1$, then $\tau(G)=\yng(3,2)$ and $G=W_3^1\times_kW_2^1$. 

Using this, we obtain a sort of converse of Proposition \ref{faithrationalactions} in the case of group schemes of height one. The following Proposition shows that if $G_1\times_k\dots\times_kG_l$ acts on a variety $X$ and the action restricted to every $G_i$ is generically free, then there exists a generically free $G$-action on $X$. 

\begin{prop}\label{arbitraryproducts}
Let $k$ be perfect, $H=\prod_{i=1}^lG_i$  be an infinitesimal commutative unipotent $k$-group scheme of height one and $X$ be a $k$-variety of dimension $n$. Then there exists a faithful rational $H$-action on $X$ which induces generically free $G_i$-actions for every $i=1,\dots,l$ if and only if there exists an infinitesimal commutative unipotent $k$-group scheme $G$ of height one such that $\dim_k(\Lie (G))\leq n$ and $G_i\hookrightarrow G$ for all $i=1,\dots,l$.
\end{prop}

\begin{proof}
     The 'if' part is clear by Remark \ref{genfreermk}. Suppose now that there exists a faithful rational $H$-action on $X$ which induces generically free $G_i$-actions for every $i=1,\dots,l$ and let $K$ denote the function field of $X$. By assumption, every $G_i$ is of height one and thus corresponds to a Young diagram  $\tau(G_i)=(n_{1i},\dots,n_{s_ii})$ for some $s_i\geq 1$ and $n_{s_ii}\neq0$. Recall that $s_i$ corresponds to the length of the first column of $\tau(G_i)$, that is $\dim_k(\Lie (\soc(G_i))$ (see Lemma \ref{younglemma}). The $H$-action is determined by a set of derivations $D_{ji}$, with $i=1,\dots l$ and $j=1,\dots, s_i$, such that they commute pairwise and $D_{ji}^{p^{n_{ji}}}=0$.
    The fact that each $G_i$-action is generically free is equivalent to the fact that $$S_i=\left\{D_{ji}^{p^{n_{ji}-1}} \mid  j=1,\dots, s_i\right\}$$ is linearly independent over $K$ for any $i=1,\dots, l$. Indeed $S_i$ represents the action induced by $\soc(G_i)$.
     Let $G$ be the smallest infinitesimal commutative unipotent group scheme of height one containing $G_i$ for all $i$. Then $\tau(G)=(n_1,\dots,n_s)$ where $s=\max\{s_1,\dots,s_l\}$ and $n_j=\max\{n_{j1},\dots,n_{jl}\}$ for every $j=1,\dots,s$. We also fix a function $$f:\{1,\dots,s\}\mapsto \{ 1,\dots,l\}$$ such that $n_j=n_{jf(j)}$. This means that for the $j$-th line of the Young diagram of $G$ we are choosing the $j$-th line of $G_{f(j)}$.
    Now we want to construct an action of $G$ on $X$, or equivalently a set of derivations $E_i$ which commute pairwise and such that $E_{i}^{p^{n_i}}=0$ for any $i=1,\dots,s$. 
    We define $E_1:=D_{1f(1)}$. Now suppose we have defined $E_r$, with $1\le r\le s-1$, such that the set $$C_r=\left\{E_k^{p^{n_k-1}}\mid k=1,\dots,r\right\}$$ is linearly independent over $K$, then we define $E_{r+1}$ in such way that it does not belong to the space generated by $C_r$. We remark that $\tau(G_{f(r+1)})$ has at least $r+1$ lines which have at least $n_{r+1}$ squares. Now $$\left\{D_{k f(r+1)}^{p^{n_k-1}}\mid k=1,\dots,r+1\right\}$$ is a set of $r+1$ $K$-linearly independent derivations, therefore there exists $k_0\in\{1,\dots,r+1\}$ such that $D_{k_0 f(r+1)}^{p^{n_{k_0}-1}}$ does not belong to the $K$-vector space generated by $C_r$.
    We define $E_{r+1}:= D_{k_0 f(r+1)}^{p^{n_{k_0}}-p^{n_{r+1}}}$. Its order is $p^{n_{r+1}}$. Therefore we constructed an action of $G$ on $X$.
    By construction we have that the set $$\left\{E_i^{p^{n_i-1}}\mid i=1,\dots,s\right\}$$ is $K$-linearly independent. This set corresponds to the induced action of the socle of $G$. Hence the action of the socle of $G$ is generically free, and the same is true for the action of $G$ by Proposition \ref{kerfrob}. 
    This implies, by Proposition \ref{dimlie}, that $\dim_k(\Lie(G))\le n$,  as wanted.
\end{proof}

\begin{rmk}
Notice that actually in the above proof we never used the fact that the $H$-action was faithful. Moreover we remark that the condition on the existence of such actions is purely combinatorial and it is equivalent to asking, using the notation of the proof, that $\dim_k(\Lie (G))=\sum_{j=1}^sn_j\le n$. For example, if we take $G_1$ and $G_2$ corresponding respectively to $$\yng(3,1)\quad\mbox{ and }\quad\yng(2,2)$$ then $$G=\yng(3,2)$$ and the Proposition implies that even if there exist generically free actions of $G_i$ on every variety of dimension $4$, there is no action of $G_1\times_kG_2$ on a variety of dimension $4$ which is generically free when restricted to $G_i$ for $i=1,2$. On the other hand, there exist such actions on every variety of dimension $\geq5$.
\end{rmk}

We finish the paper illustrating the above results in the case of the connected part of the $p$-torsion of abelian varieties.

\begin{example}
    Let $k$ be algebraically closed and $A$ be an
    abelian variety defined over $k$ of dimension $g$, $p$-rank $f$ and $a$-number $a$.
    If there exists a faithful rational action of $A[p]^0$ on a curve, then by Proposition \ref{actionkerF}, $f\leq1$ and either $A[p]^{0,u}$ is trivial (if $f=1$) or $V_{\ker(F_{A[p]^{0,u}})}=0$ (if $f=0$). In either case, it holds $\soc(A[p]^0)=\ker(F_{A[p]}).$ As a consequence one has that $a+f=g$. We then have the following two cases.
     \begin{itemize}
        \item If $f=1$, then $f=1=g$, that is $A$ is an ordinary elliptic curve and faithful rational actions of $A[p]^0=\mu_p$ on any curve always exist. 
        \item If $f=0$, then $a=g$ that is $A$ is a superspecial abelian variety.  Superspecial abelian varieties are always isomorphic to products of supersingular elliptic curves \cite[Theorem 2]{OortSSpAV}. In Example \ref{counterexample} we saw that there is no faithful rational action of $E[p]\times_kE[p]$ on any curve, for $E$ supersingular. 
    \end{itemize}
    Therefore, we can conclude that there exists a faithful rational action of $A[p]^0$ on a curve if and only if $A$ is an elliptic curve.
    More generally, if there exists a faithful rational action of $A[p]^0$ on a variety of dimension $n$, then $0\leq g-f\leq a(n-f)$.
    Indeed, by Proposition \ref{actionkerF}, we have $f\leq n$ and $V^{n-f}_{\ker(F_{A[p]^{0,u}})}=0$ (if $f=n$ then there is no unipotent part). This means that $$\ker(F_{A[p]})\simeq\prod_{i=1}^aW_{n_i}^1\times_k\mu_p^f$$ where $n_i\leq n-f$ for every $i\in I$. As a consequence, $g-f=\sum_{i\in I}n_i\leq a(n-f)$. 
    
    Notice that if $g\leq n$ we don't get any interesting information and moreover by Remark \ref{thm+Brion} there exist always generically free rational actions of $A[p]^0$ on varieties of dimension $n$. Nevertheless, such faithful rational actions may occur even when $g>n$ (if $n>1$, as seen in the first part). For example, by Proposition \ref{suffcondfaithact}, there exist faithful rational actions of the $p$-torsion of a superspecial abelian variety of any dimension on any variety of dimension $\geq2$ 
    (but not on curves). 
    
    The numerical condition $g-f\leq a(n-f)$ holds true for any $G\simeq G^u\times_kG^d$ infinitesimal commutative trigonalizable $k$-group scheme with a faithful rational action on a variety of dimension $n$, with $a=\dim_k(\Lie (\soc(G^u)))$, $f=\dim_k(\Lie (\soc(G^s)))$ and $g=\dim_k(\Lie (G))$.
\end{example}

\printbibliography
\noindent\textsc{Institut de Math\'ematiques de Bordeaux, 351 Cours de la Lib\'eration, 33405 Talence, France}\\\textit{Email Address:} \textbf{bianca.gouthier@math.u-bordeaux.fr}\\
\noindent\textsc{Mathematical Institute, Heinrich-Heine-University, Universitätsstr. 1, 40225 Düsseldorf, Germany}\\
\textit{Email Address:} \textbf{bianca.gouthier@hhu.de}
\end{document}